\documentclass[11pt,a4paper]{amsart}
\usepackage{graphicx}
\usepackage{amsmath}
\usepackage{amssymb}
\usepackage{amsfonts}
\usepackage{amsthm}
\usepackage{xcolor}
\usepackage{caption}

\usepackage{tikz}
\usepackage{pgfplots}
\pgfplotsset{compat=newest}
\usetikzlibrary{3d,calc}

\usepackage{tikz}
\usepackage{tikz-3dplot}

\usepackage[text={15cm,21cm},centering]{geometry}
\def\neweq#1{\begin{equation}\label{#1}}
	\def\endeq{\end{equation}}

\def\R{\mathbb{R}}

\def\eps{{\varepsilon}}

\newtheorem{proposition}{Proposition}
\newtheorem{theorem}[proposition]{Theorem}
\newtheorem{lemma}[proposition]{Lemma}

\theoremstyle{remark}
\newtheorem{remark}[proposition]{Remark}

\theoremstyle{definition}
\newtheorem{definition}[proposition]{Definition}

\numberwithin{equation}{section}
\numberwithin{proposition}{section}

\usepackage{hyperref}
\hypersetup{
	colorlinks=true,
	linkcolor=blue,
	filecolor=magenta,
	urlcolor=cyan,
}

\title[Equilibrium configurations of a 3D fluid-beam interaction problem]{Equilibrium configurations of a\\ 3D fluid-beam interaction problem}
	
	\author{Vincenzo Bianca, Edoardo Bocchi and Filippo Gazzola}
    \address{
    Dipartimento di Matematica, Politecnico di Milano, Piazza Leonardo da
Vinci 32, 20133 Milano, Italy}
\email{vincenzo.bianca@polimi.it, edoardo.bocchi@polimi.it, filippo.gazzola@polimi.it}

\begin{document}

	\begin{abstract} We study a fluid-structure interaction problem between a viscous incompressible fluid and an elastic beam with fixed endpoints in a static setting. The 3D fluid domain is bounded, nonsmooth and non simply connected,
the fluid is modeled by the stationary Navier-Stokes equations subject to inflow/outflow conditions. The structure is modeled by a stationary 1D beam equation with a load density involving the force exerted by the fluid and, thereby, may vary its position. In a smallness regime, we prove the existence and uniqueness of the solution  to the PDE-ODE coupled system. \bigskip

\emph{AMS Mathematics Subject Classification:} 35Q35, 76D05, 74F10.\par
\emph{Keywords and phrases:} Fluid-structure interaction, Navier-Stokes equations, beam equation, PDE-ODE system, existence and uniqueness of equilibria.
\end{abstract}
	\date{}
	\maketitle
	
	\section{Introduction}
	We consider a fluid-structure interaction (FSI) problem in the rectangular parallelepiped
	$$P=(-R,R)\times(-1,1)^2,\hspace{.3in}R\gg1$$
	representing a wind tunnel in which artificial airflows are set at the inlet and outlet. The wind tunnel is crossed by a beam occupying a
	compact connected set
	$\mathcal{B}$ such that $\Omega = P\setminus\mathcal{B}$ is not simply connected. The fluid is viscous, homogeneous, incompressible and governed by the stationary Navier-Stokes equations
	\neweq{SNS}
	-\eta\Delta u+u\cdot\nabla u+\nabla p=0,\quad\nabla\cdot u=0\quad\mbox{in} \quad \Omega,
	\endeq
	where $u=(u_1,u_2,u_3)\in\R^3$ is the velocity vector field, $p$ is the scalar pressure and $\eta>0$ is the kinematic viscosity. The fluid is moving due to inflows/outflows located at the faces $\{x=\pm R\}$.
	The beam $\mathcal{B}$ is either {\em hinged or clamped} to $\partial P$ at $y=\pm1$ and its cross-sections $\mathcal{B}_y$ (for $-1\le y\le1$)
are usually irregular hexagons with
	rounded corners, see Figure \ref{dom1}.
	
	\begin{figure}[ht]
		\begin{center}
			\includegraphics[scale=0.25]{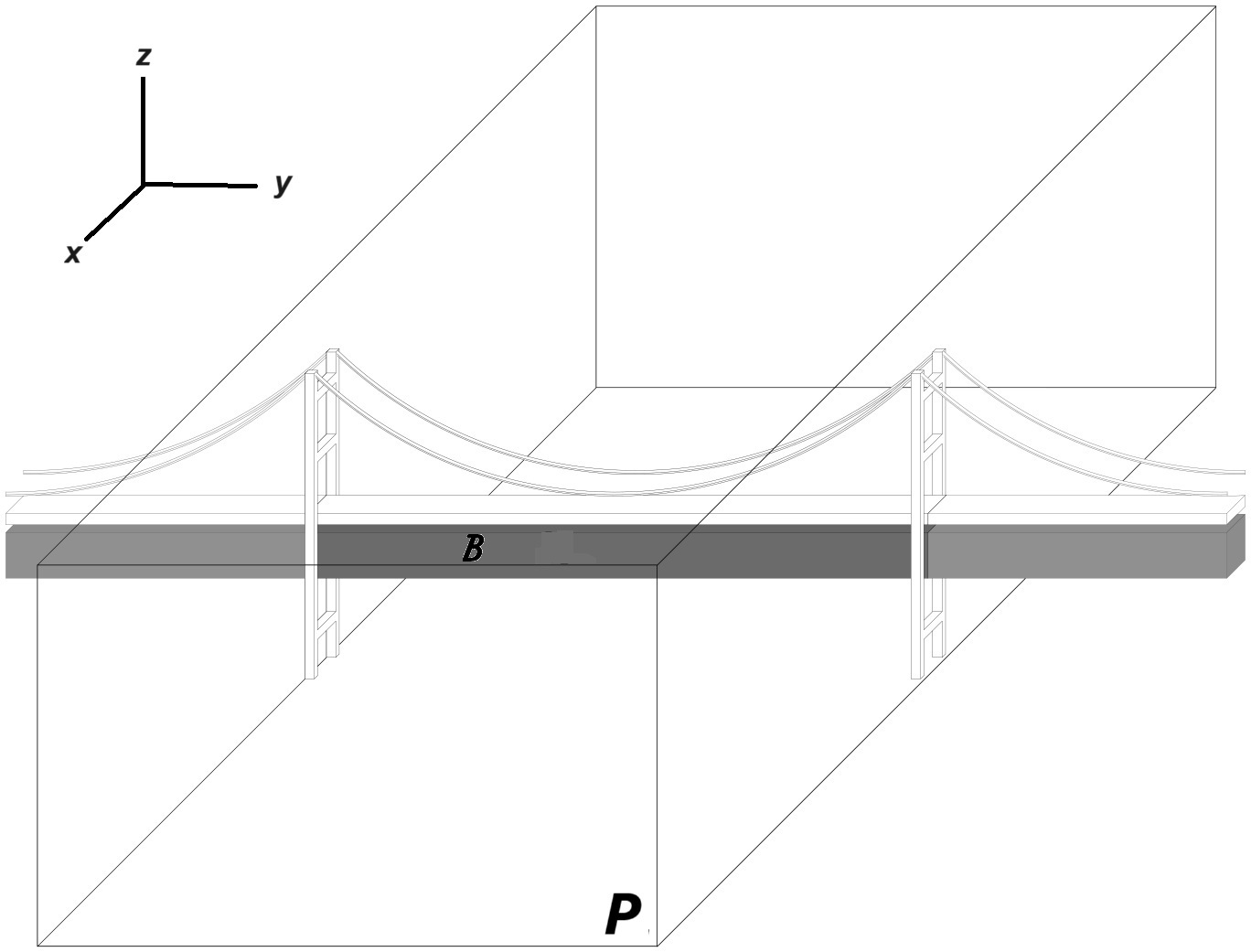}\qquad\qquad\includegraphics[scale=0.36]{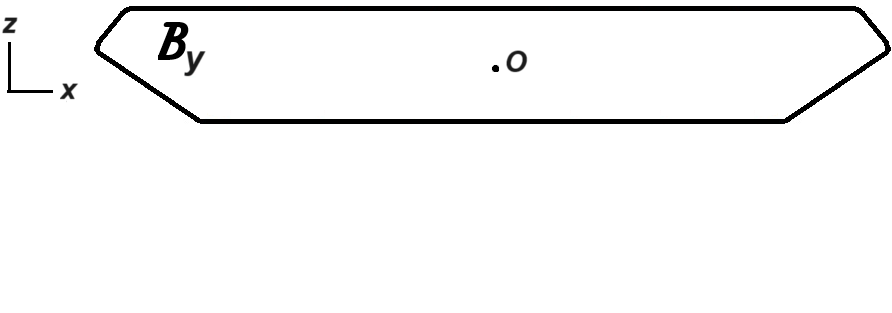}
		\end{center}
		\caption{The fluid domain $\Omega=P\setminus\mathcal{B}$ and a cross-section $\mathcal{B}_y$.}\label{dom1}
	\end{figure}
	
This model has natural applications in the stability of suspension bridges under the action of the wind \cite{denisgaldi3,irena,gazbook}.
The bridge is divided in three adjacent spans, separated by two fixed piers and we focus our attention on the
most vulnerable center span, see the left picture in Figure \ref{dom1}. \par
The wind hits the deck horizontally which modifies the flow that goes around the deck: in the (downwind) leeward, the flow creates
asymmetric vortices that
generate an orthogonal force $L=L(u,p,y)$ named {\em lift} \cite{Acheson1990, PaPriDeLan10}, which triggers the vertical displacements of the deck;
see the experiment in \cite{wolfgang4}. Therefore, $\mathcal{B}$ moves vertically and reaches a new position $\mathcal{B}^h$, where
$h=h(y)$ is a smooth function describing its vertical displacement. This also changes the fluid domain and \eqref{SNS} is satisfied in
$\Omega_h=P\setminus\mathcal{B}^h$. Further forces (e.g., the drag and torque) may also appear, see
again \cite{Acheson1990, PaPriDeLan10}, but here we do not consider them. The resulting movements of the deck create more perturbations of the flow
and generate a full FSI.\par
In absence of a flow, the deck $\mathcal{B}$ is maintained in its equilibrium position ($h\equiv0$) by gravity combined with some
elastic forces due to the cables, hangers, and bending of the deck so that the overall forces $f=f(h)$ have the same sign as $h$.
In the static setting inside a wind tunnel, the vertical position $z=h(y)$ of the barycenter of the cross-section $\mathcal{B}_y$ obeys to the beam-type equation
	\begin{equation}\label{compatibility}
		h''''(y)+f\big(h(y)\big)=L(y, u, p)=L\big(y, h(y)\big)
	\end{equation}
	complemented with some boundary conditions, see Section \ref{sec_geometric}. In \eqref{compatibility} the lift $L$
	also depends on $h$ since $(u,p)$ solves \eqref{SNS} in $\Omega_h$.\par
	The equilibrium configurations of the FSI problem are obtained by solving
	the {\em system} \eqref{SNS}-\eqref{compatibility}.
	It is our purpose to analyse existence, uniqueness, stability of the equilibrium configurations of
	\eqref{SNS}-\eqref{compatibility} as the inflows/outflows vary. In \cite{simona, edofil,edofil2,GS} the 2D version of this problem was investigated, where the structure represents the cross-section of the deck. The associated dynamical setting was studied in \cite{BHPS, pata, Patriarca2022}.
	In 3D, the only known result \cite{clara} is that, when $\mathcal{B}$ is a rectangular plate and under {\em symmetry} assumptions on both
	$\Omega$ and the inflow/outflow (Poiseuille type), for small flows the unique equilibrium position remains $\Omega$; see also \cite{denisgaldi, FGS2020} for related problems. These contributions
are connected to one of the classical Leray's problems \cite{AM2,AM3}.
 Other interactions between fluids and elastic structures were extensively studied in the last decade, both mathematically \cite{KukaMu,Bociu20,GraHil16,MuS22} and from an engineering point of view \cite{Kanaris-Grigoriadis-Kassinos2011}, see also references therein. \par	
	In Section \ref{sec_geometric} we describe in full detail the geometric and physical frameworks, also giving
	the precise assumptions used in the sequel.\par
In Section \ref{fluid-section} we discuss existence, uniqueness, and a priori bounds for
	the solutions to \eqref{SNS}, {\em assuming that $h$ is given}, see Theorem \ref{the_ex_regu}. The difficult task here is the choice of the functional spaces for dealing with {\em strong solutions}, see Definition
\ref{def_1}. We need enough regularity to have embedding into boundary spaces where the (1D) lift $L=L(y,u,p)$ is well-defined, see \eqref{lift} below. Since all the $\Omega_h$ are merely Lipschitz (and nonconvex),
to gain regularity we first perform a reflection with respect to some faces of $P$. Moreover, at the other angular points of $\partial\Omega_h$, since {\em we cannot increase regularity} in Hilbertian Sobolev spaces as in the classical works by Grisvard \cite{GrisvardBUMI,grisvardpisa,grisvardlibro}, we {\em increase
integrability} by switching to non-Hilbertian spaces as in the monograph by Maz'ya \& Rossmann \cite{Mazya-Rossmann2010}, without falling in the more delicate world of weighted spaces. \par
In Section \ref{sec_FSI} we take advantage of this functional framework and we tackle the
FSI problem \eqref{SNS}-\eqref{compatibility} in $\Omega_h$; we prove that in the uniqueness regime for \eqref{SNS} (small flows) and in a neighbourhood
of $h=0$ there exists a unique equilibrium configuration that varies with Lipschitz continuity with respect to the flows, see Theorem \ref{main-theo}.
The proof is obtained with a fixed-point procedure that requires the introduction of suitable functional spaces also for the beam displacement $h$.
By exploiting Theorem \ref{the_ex_regu}, we first show that for a given $\overline{h}$ there exists a unique solution $(u(\overline{h}), p(\overline{h}))$ to \eqref{the_ex_regu} satisfying some bounds. Then we prove that, for such $(u(\overline{h}), p(\overline{h}))$, there exists a unique (strong) solution $h$ to \eqref{compatibility} with $L(y, \overline{h})$ under the ``physical'' boundary conditions described in Section \ref{sec_geometric}.
Finally, through a contraction principle we find a unique fixed point $(h, u,p)$, solving \eqref{SNS}-\eqref{compatibility}. We conclude the paper with two results in a symmetric framework, see Theorems \ref{symmNS} and \ref{main-theo-symm}, and give an explanation of the consequent physical phenomenon.
	
	\section{Geometric and physical frameworks}\label{sec_geometric}
	
	In the sequel, we denote by
	\begin{equation*}
		\mathbb{H}\mbox{ either the space }H_0^2(-1,1)\mbox{ or }H^2\cap H_0^1(-1,1).
	\end{equation*}
	It is well-known \cite[Theorem 2.31]{GGSbook} (valid for {\em any} bounded Lipschitz domain in $\R^n$ and for {\em any} $n\ge1$), that
	 $\mathbb{H}$ is a Hilbert space endowed with the scalar product and norm
	\begin{equation}\label{Hnorm}(h_1, h_2)_{\mathbb{H}}= \int_{-1}^1 h_1'' h_2'', \qquad   \|h\|_\mathbb{H}=\|h''\|_{L^2(-1,1)}\sim \|h\|_{H^2(-1,1)}. \end{equation}
Moreover, we have the embedding inequality
\begin{equation}\label{SH-embedding}
\|h\|_{L^\infty(-1,1)}\leq S_{\mathbb{H}}\|h\|_{\mathbb{H}} \quad \text{where} \quad \frac{1}{S_{\mathbb{H}}}=\min_{h\in\mathbb{H}}\, \frac{\|h\|_{\mathbb{H}}}{\|h\|_{L^\infty(-1,1)}},
    \end{equation}see Appendix for an upper bound of $S_\mathbb{H}$.
Analogously, we denote by
\begin{equation*}
		\mathbb{H}^4\mbox{ either the space $H^4\cap H^2_0(-1,1)$ or $\{h\in H^4(-1,1)$ s.t. $h(\pm 1)=h''(\pm1)=0 \}$}
	\end{equation*} endowed with the scalar product and norm
	\begin{equation}\label{H4norm}(h_1, h_2)_{\mathbb{H}^4}= \int_{-1}^1 h_1'''' h_2'''', \qquad   \|h\|_{\mathbb{H}^4}=\|h''''\|_{L^2(-1,1)}\sim \|h\|_{H^4(-1,1)},\end{equation} see again Appendix. But \eqref{H4norm} \emph{is not} a norm in $H^4 \cap H^1_0(-1,1)$: to see this, note that $h(y)=1-y^2 \in H^4 \cap H^1_0(-1,1)$ has null $\mathbb{H}^4$-norm.\par
	Consider a $C^{1,1}$ compact, connected and simply connected set $\mathcal{B}^*\subset\mathbb{R}^3$ such that
	\begin{equation}\label{B*}\begin{array}{c}
			X^+=\max_{(x,y,z)\in\mathcal{B}^*}x<R, \quad X^-=-\min_{(x,y,z)\in\mathcal{B}^*}x<R,\\[5pt]
			
			Z^+=\max_{(x,y,z)\in\mathcal{B}^*}z<1, \quad Z^-=-\min_{(x,y,z)\in\mathcal{B}^*}z<1 	\\[5pt]
			
			\max_{(x,y,z)\in\mathcal{B}^*}y>1,\hspace{.3in} \min_{(x,y,z)\in\mathcal{B}^*}y<-1,\hspace{.3in}\mathcal{H}^2(\mathcal{B}_y^*)>0
		\end{array}
	\end{equation}
	where $\mathcal{H}^2$ denotes the $2$-dimensional Hausdorff measure and
	$
	\mathcal{B}_y^*=\mathcal{B}^*\cap(\mathbb{R}\times\{y\}\times\mathbb{R})
	$
	is the cross-section of $\mathcal{B}^*$, namely the intersection with a plane parallel to both the $x$ and $z$ axes. Let
	$$\mathcal{B}=\overline{P}\cap\mathcal{B}^*$$be the portion of $\mathcal{B}^*$ contained in $P$,
	so that its boundary $\partial\mathcal{B}$ consists of three parts
	\neweq{Bboundary}
	\partial\mathcal{B}=\mathcal{B}^*_{-1}\cup \mathcal{B}^*_1\cup\mathcal{S}
	\endeq
	where $\mathcal{S}$ is a $C^{1,1}$-surface that does not intersect the faces $\{y=\pm1\}$. For any $y\in[-1,1]$ we simply denote
	$\mathcal{B}_y=\mathcal{B}_y^*$. For any $h\in\mathbb{H}$, the non-simply connected domain
	\neweq{Omegah}
	\Omega_h=P\setminus\mathcal{B}^h={P}\setminus\{(x,y,z+h(y))\mbox{ with }(x,y,z)\in\mathcal{B}\}
	\endeq
	is the region occupied by the fluid; in particular, $\mathcal{B}^0=\mathcal{B}$ and $\Omega_0=\Omega$. We assume that
	\begin{equation*}
	\begin{array}{cc}
		\mathcal{B}_y\subset(-R,R)\times(-1,1)\mbox{ is a connected planar compact set}\\
		\mbox{with barycenter at $(x,z)=(0,0)$ for all $|y|\le1$,}
	\end{array}
	\end{equation*}
	see again Figure \ref{dom1} (right). Then, in view of \eqref{Omegah}, the barycenter of the cross-section $\mathcal{B}_y^h$ is located at
	$(x,z)=(0,h(y))$. In Figure \ref{duebeams}, to simplify the pictures we sketch the domains $\Omega$ and $\Omega_h$ for {\em nonsmooth}
	(rectangular) cross-sections.
	
	\begin{figure}[ht]
		\begin{center}
			\includegraphics[scale=0.4]{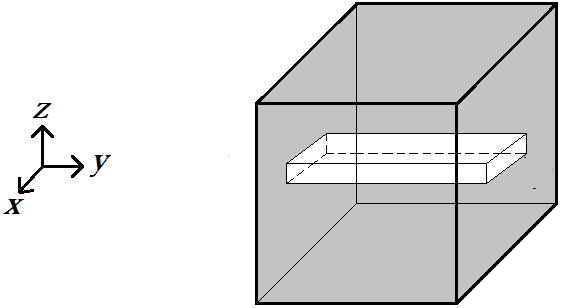}\qquad\qquad\includegraphics[scale=0.4]{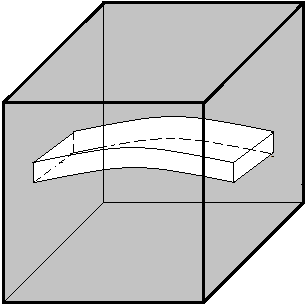}
		\end{center}
		\caption{Examples of fluid domains $\Omega$ (left) and $\Omega_h$ (right).}\label{duebeams}
	\end{figure}
	
	We denote the faces of the parallelepiped $P$ by
	\begin{equation*}
		\Gamma_x^-=\{-R\}\times(-1,1)\times(-1,1), \qquad \Gamma_x^+=\{R\}\times(-1,1)\times(-1,1),
	\end{equation*}
	\begin{equation*}
		\Gamma_y^-=(-R,R)\times\{-1\}\times(-1,1),\qquad 	\Gamma_y^+=(-R,R)\times\{1\}\times(-1,1),
	\end{equation*}
	\begin{equation*}
		\Gamma_z^-=(-R,R)\times(-1,1)\times\{-1\}, \qquad \Gamma_z^+=(-R,R)\times(-1,1)\times\{1\}.
	\end{equation*}
	Hence, $\partial\Omega_h$ consists of the two closed squares $\overline{\Gamma_x^\pm}$, the two closed rectangles $\overline{\Gamma_z^\pm}$
	and the two closed non-simply connected planar regions $\overline{\Gamma_y^\pm\setminus\mathcal{B}_{\pm1}}$.
	After putting $\varepsilon^\pm(h)=\mathrm{dist}(\mathcal{B}^h, \Gamma_z^\pm)$, it follows from \eqref{B*} that
	\begin{equation*}  \ \mathcal{B}^h \subset [-X^-,X^+]\times [-1,1]\times[-1+ \eps^-(h), 1-\eps^+(h)]. \end{equation*}
Although the boundary of $\mathcal{B}^*$ is assumed to be smooth ($C^{1,1}$), the domain $\Omega_h=P\setminus\mathcal{B}^h$ is {\em merely Lipschitz}
	not only because of the vertices and edges of the parallelepiped $P$ but also because of the intersection between $\mathcal{B}^h$ and the faces
	$\Gamma_y^\pm$ of $P$. In fact, a (nonsmooth) vertical displacement $h=h(y)$ could also affect the regularity of $\mathcal{S}$, see \eqref{Bboundary},
	but, as we shall see, a posteriori $h$ will be sufficiently smooth.\par	

	Since $\mathbb{H}$ compactly embeds into $C([-1,1])$,  we have the lower bounds
	\begin{equation}\label{bound-eps+-}\begin{aligned}
			&\varepsilon^+(h)\geq 1-\max_{[-1,1]} h -Z^+ \geq 1- \|h\|_{L^\infty(-1,1)} -Z^+\\
			&\varepsilon^-(h)\geq1+\min_{[-1,1]} h -Z^-\geq 1- \|h\|_{L^\infty(-1,1)} - Z^-
		\end{aligned}\qquad\forall h\in\mathbb{H}.
	\end{equation}
A sufficient condition to prevent collisions (intersections) between $\mathcal{B}^h$ and the top/bottom boundaries of $P$, namely $\eps^\pm (h)=0$, is then that there exists \begin{equation*}
1\leq \omega<\frac{1}{\max\{Z^+,Z^-\}}
\end{equation*} such that
\begin{equation}\label{eq_bound_infty}
\left\|h\right\|_{L^\infty(-1,1)} <	1- \omega\max\{Z^+,Z^- \}.\\[5pt]
\end{equation}

We {\em do not assume} that all the cross-sections $\mathcal{B}_y$ have the same shape, such as for the suspension bridge over the Grand Canyon in
	Zhangjiajie (China), see Figure \ref{dom2}.
	\begin{figure}[ht]
		\begin{center}
			\includegraphics[scale=0.36]{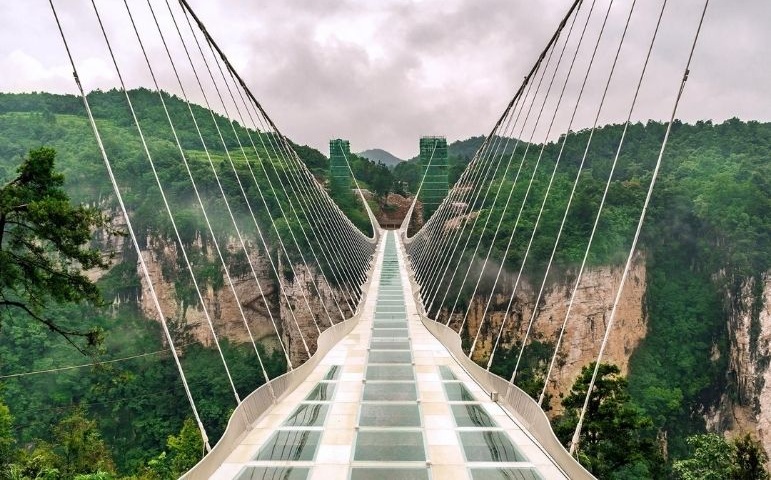}
		\end{center}
		\caption{The Zhangjiajie Bridge.}\label{dom2}
	\end{figure}
	For the sake of stability, the deck of ``asymmetric'' bridges should have wider cross-sections at the extremals $y=\pm1$ and smaller
	cross-sections in proximity of the center $y=0$ of the span. This is why the angles of attachment between $\mathcal{B}$ and $\Gamma_y^\pm$
	may be larger than the right angle. To make this precise, consider $A=(x,y,z)\in\partial\mathcal{B}_{\pm1}$, namely a point where
	$\mathcal{B}^h$ is attached to the faces $\Gamma_y^\pm$ of $P$. For any $A\in\partial\mathcal{B}_1$
	(same construction for $A\in \partial \mathcal{B}_{-1}$) and $h\in\mathbb{H}$
	there exists $\alpha\in[0,2\pi)$ such that $A$ belongs to the curve
	\begin{equation}\label{Palpha}
	\mathbb{C}^{h,\alpha}=\partial\mathcal{B}^h\cap\mathcal{P}_\alpha\quad\mbox{where}\quad
	\mathcal{P}_\alpha=\{(x,y,z)\in\R^3;\, (\cos\alpha)x+(\sin\alpha)z=0\}.
	\end{equation}
	Within the plane $\mathcal{P}_\alpha$, the tangent line to $\mathbb{C}^{h,\alpha}$ forms an angle of measure $\theta=\theta(A)$ with the line
	$\{y=1\}\cap\mathcal{P}_\alpha$, see Figure \ref{angle}.
	\begin{figure}[ht]
		\begin{center}
			\includegraphics[scale=0.5]{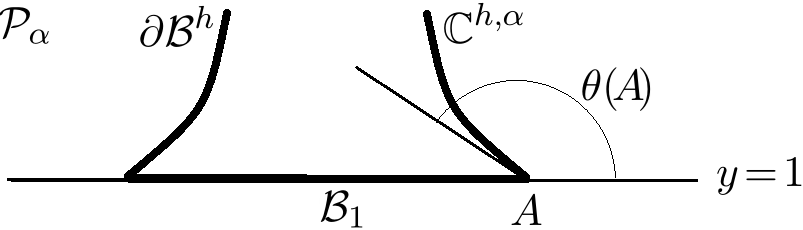}
		\end{center}
		\caption{The curve $\mathbb{C}^{h,\alpha}$ in the plane $\mathcal{P}_\alpha$.}\label{angle}
	\end{figure}
	
	The angle $\theta(A)$ may depend on both $\alpha\in[0,2\pi)$ and $h\in\mathbb{H}$. By continuity of $A\mapsto \theta(A)$ and compactness of  $\partial \mathcal{B}_{\pm1} $, we infer that
	\begin{equation}\label{eq_theta_*}
		\theta_*=\sup_{\|h\|_{\mathbb{H}}< C_\mathbb{H}} \max_{\alpha\in [0,2\pi]} \max_{A\in\partial \mathcal{B}_{\pm1}}\theta(A)<\pi
	\end{equation}
	where $C_\mathbb{H}$ is sufficiently small to ensure \eqref{eq_bound_infty} in view of \eqref{SH-embedding}. Mathematically, any $\theta^*\in (0, \pi)$ is allowed but the case $\theta^*\in (0, \pi/2)$ does not physically describe a real bridge. The case $\theta(A)=\pi/2$ may occur, for instance if $A$ belongs to the upper part of the deck (where cars or pedestrians move).

From now on, with the convention $\tfrac{1}{0}=\infty$, we fix $\sigma$ such that    \begin{equation}\label{eq_sigma}
    \text{if}\  0< \theta_* < \frac{\pi}{2}, \quad
    0<\sigma < 1, \qquad  \text{if} \ \frac{\pi}{2}\leq \theta_* < \pi,  \quad
    0<\sigma < \min\left\{\frac{2\pi-2\theta_*}{2\theta_*-\pi},1\right\}.
\end{equation} The choice of $\sigma$ is fundamental for the regularity of the solution to the Navier-Stokes equations, see Definition \ref{def_1}, and for the definition of the lift that requires one-dimensional traces in the fluid domain, see \eqref{lift}. We complement \eqref{SNS} with inflow/outflow conditions at the inlet/outlet $\Gamma_x^\pm$.
	Given $$\begin{array}{c}V\in W^{2-\tfrac{1}{2+\sigma}, 2+\sigma }(-1,1)^2\subset C([-1,1]^2)  \\[5pt] \text{such that} \quad  V(y,-1)=0 \ \text{for} \ |y|\le1,  \quad V(\pm1 ,z)=0 \ \text{for} \ |z|\le1,\end{array}$$ we take $$V_i, V_o \in W^{2-\tfrac{1}{2+\sigma}, 2+\sigma }(\Gamma_x^\pm)$$ verifying
	\begin{align}
		&\quad \quad V_i=V_o=V\quad \mbox{on }\quad \overline{\Gamma_{y}^\pm \setminus \mathcal{B}_{\pm 1}}\cup\Gamma_z^\pm, \qquad
		\int_{\Gamma_x^-} V_i=\int_{\Gamma_x^+} V_o=1.\label{flux}
	\end{align}
	For $\gamma>0$ we consider the boundary-value problem
\begin{equation}\label{BVprobRR}
	\begin{aligned}
		&-\eta\Delta u+u\cdot\nabla u+\nabla p=0,\quad\nabla\cdot u=0\quad\text{in}\quad \Omega_h\\[2pt]
		&u=0 \quad \text{on} \quad  \partial \mathcal{B}^h ,\quad u=\gamma Ve_1\quad \text{on} \quad \overline{\Gamma_{y}^\pm \setminus \mathcal{B}_{\pm 1}}\cup\Gamma_z^\pm\\[2pt]
		&u= \gamma V_ie_1\quad \text{on} \quad \Gamma_x^-,\quad u= \gamma V_oe_1\quad\text{on} \quad \Gamma_x^+.
	\end{aligned}
\end{equation}
	Since the fluid is incompressible, the flux $\gamma>0$ is constant through any cross-section $\{x=x_0\}$ (for $|x_0|\le R$) of $\Omega_h$:
	this is why \eqref{flux} is a {\em necessary condition} for the existence of solutions to \eqref{BVprobRR}.
Moreover, \eqref{Omegah} shows that the volumes of $\Omega$ and $\Omega_h$ coincide, as necessary for incompressible fluids. Finally, in \eqref{BVprobRR} and all the subsequent equations, the pressure $p$ is defined up to the addition of a constant
and, without further mention, we will always assume that {\it the mean value over $\Omega_h$ of the involved pressure is zero}.\par	
	For a solution $(u,p)$ to \eqref{BVprobRR}, the lift force acting vertically (in the $z$-direction) on a cross-section $\mathcal{B}_y^h$ for a given
	$y\in[-1,1]$ is computed through the formula
	\begin{equation}\label{lift}
		L(y, u,p)= -e_3\cdot\int_{\partial\mathcal{B}_y^h} \mathbb{T}(u, p)n= -e_3\cdot\int_{\partial\mathcal{B}_y^h}[\eta(\nabla u+\nabla^Tu)-p\mathbb{I}]n,
	\end{equation}
	where $\mathbb{I}$ is the $3\times3$ identity matrix and $n$ is the outward unit normal vector to $\partial\mathcal{B}^h$, pointing inside $\mathcal{B}^h$.
	Since $\mathcal{B}^*$ is of class $C^{1,1}$, the normal vector $n$ is well-defined and, as we shall see, the solution
	$(u,p)$ is smooth enough to ensure that  the integrand in \eqref{lift} is bounded.\par
	To model the behaviour of $\mathcal{B}$, we set up a one-dimensional horizontal thin beam covering the interval $(-1,1)$.
	If $h(y)$ denotes the vertical displacement of the beam (in the $z$-direction)
	from the equilibrium position ($\equiv0$) at the point $y\in(-1,1)$, from \cite[Section 1.1.1]{GGSbook} we know that the energy
	density $E_b$ stored by {\em bending} the beam is proportional to the square of the curvature:
	\begin{equation}\label{Jb-gs}
		E_b(h)=\frac{EI}{2}\int_{-1}^{1}\frac{h''(y)^{2}}{\left(1+h'(y)^{2}\right) ^{3}}\sqrt{1+h'(y)^{2}}\, dy=
		\frac{EI}{2}\int_{-1}^{1}\frac{h''(y)^{2}}{\left(1+h'(y)^{2}\right)^{5/2}}\, dy,
	\end{equation}
	where $EI>0$ is the flexural rigidity. Formula \eqref{Jb-gs} highlights the curvature and the arclength.
	The functional $E_b$ {\em does not include the stretching} energy, a term measuring the length variation of the beam
	since we consider beams free to move in horizontal directions at their endpoints, as at the piers in suspension bridges.
	For small displacements of the beam, in particular for small $h'$, the approximated bending energy reads
	\neweq{approxenergy}
	E_b(h)\approx\int_{-1}^{1}\frac{h''(y)^{2}}2 \, dy,
	\endeq
	in which we normalised the flexural rigidity.\par
	The beam is subject to restoring forces $f=f(h)$ that satisfy
	\neweq{assf}
	f\in{\rm Lip}(\R)\, ,\quad f(0)=0\, ,\quad h\mapsto f(h)\mbox{ is non-decreasing.}
	\endeq
	As mentioned in the introduction, these forces are due to the cables, hangers, and gravity. From a purely mathematical point of view,
	\eqref{assf} does not exclude $f(h)\equiv0$ that, instead, has to be excluded in the 2D case \cite{simona,edofil,edofil2,denisgaldi}:
the reason is that the beam is fixed at its endpoints while in 2D the structure may be ``floating''.
	Due to the fluid flow, the beam is also subject to the lift force $L=L\big(y, h(y)\big)$; in fact, $L=L(y, u,p)$, see \eqref{lift},
	where $(u,p)$ solves \eqref{BVprobRR} in $\Omega_h$ and, hence, depends on $h$. As we shall see in Lemma \ref{lem-Lbound} and Proposition \ref{L-incre}, $L$ is bounded with respect to $y$ and Lipschitz continuous with respect to $h$.  Hence, after introducing the {\em potentials} $F: \R \rightarrow \R$ and $G : (-1,1)\times \R \to \R $ such that
	\begin{equation*}\label{F-G}
		F'(h)=f(h),\qquad G_h(y,h)=L(y,h ),
	\end{equation*}
	the overall (approximated) energy of the beam is
	\neweq{totalenergy}
	\int_{-1}^{1}\frac{h''(y)^{2}}2 \, dy+\int_{-1}^{1}F\big(h(y)\big)\, dy-\int_{-1}^{1}G\big(y, h(y)\big)\, dy\qquad\forall h\in\mathbb{H}\, .
	\endeq
	Under suitable assumptions on $G$ (satisfied in our context), this functional is well-defined over $\mathbb{H}$ and its
	critical points satisfy the Euler-Lagrange equation
	$$
	\int_{-1}^{1}\Big[h''(y)\phi''(y)+f\big(h(y)\big)\phi(y)-L\big(y, h(y)\big)\phi(y)\Big]\, dy=0\qquad\forall\phi\in\mathbb{H}
	$$
	and, in strong form,
	\begin{equation}\label{compa-beam}
		h''''(y)+f\big(h(y)\big)=L\big(y, h(y)\big)\qquad \text{for a.e. } y\in(-1,1)\, ,
	\end{equation}
	that coincides with \eqref{compatibility}.
	In Section \ref{sec_FSI}, we will prove that $h\in \mathbb{H}^4$ and that \eqref{compa-beam} is satisfied in strong form, see Proposition \ref{prop-exi-beam}. However, for the regularity of solutions to \eqref{BVprobRR}, it is sufficient to consider the larger space
\begin{equation*}
C_0^{1,1}([-1,1])=\left\{h\in C^{1,1}([-1,1])\mbox{ s.t. }h(\pm1)=0\right\}.
\end{equation*}

    Finally, the differential equation \eqref{compa-beam} has to be complemented with some boundary conditions,
	depending on how the beam is fixed at the endpoints. For a beam representing a bridge, there are two possibilities:\par
	$\bullet$ Clamped: mathematically, this means that $h(\pm1)=h'(\pm1)=0$ (homogeneous Dirichlet boundary conditions), which leads to the request that $h\in H^2_0(-1,1)$. In such case, the supremum in \eqref{eq_theta_*} disappears since $\theta(A)$ does not depend on $h$.\par
$\bullet$ Hinged: mathematically, this means that $h(\pm1)=h''(\pm1)=0$ (homogeneous Navier boundary conditions), which leads to the request that $h\in H^2\cap H^1_0(-1,1)$;
		in fact, in this space $h''(\pm1)$ may not be defined and this condition should be seen as a natural boundary condition. As just mentioned, in Section \ref{sec_FSI}, we will see that $h\in \mathbb{H}^4$, which implies that $h''(\pm1 )=0$.\par
These two boundary conditions are physically represented in Figure~\ref{beamfig}.
	
	\begin{figure}[h!]
		\begin{center}
			\resizebox{10cm}{!}{\includegraphics{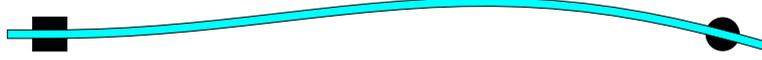}}
			\caption{The depicted boundary conditions for the endpoints is clamped (left) and hinged (right).}\label{beamfig}
		\end{center}
	\end{figure}

	\section{Fluid boundary-value problem}\label{fluid-section}

To simplify notations, we group the boundary conditions in \eqref{BVprobRR} by defining $u_*:\partial\Omega_h\to\mathbb{R}^3$ as
\begin{equation}\label{u*}
u_*=
\begin{cases}
0&\hspace{.3in}\mbox{on}\hspace{.1in}\partial\mathcal{B}^h\\
\gamma Ve_1&\hspace{.3in}\mbox{on}\hspace{.1in}\overline{\Gamma_{y}^\pm \setminus \mathcal{B}_{\pm 1}}\cup\Gamma_z^\pm\\
\gamma V_ie_1&\hspace{.3in}\mbox{on}\hspace{.1in}\Gamma_x^-\\
\gamma V_oe_1&\hspace{.3in}\mbox{on}\hspace{.1in}\Gamma_x^+.
\end{cases}
\end{equation}
Let us explain what is meant by solution to \eqref{BVprobRR}, that is,
\begin{equation}\label{eq_main}
\begin{aligned}
&-\eta\Delta u+u\cdot\nabla u+\nabla p=0,&\qquad&\nabla\cdot u=0&\qquad&\mbox{in}\hspace{.1in}\Omega_h\\
&u=u_*&\qquad& &\qquad&\mbox{on}\hspace{.1in}\partial\Omega_h.
\end{aligned}
\end{equation}
	
	\begin{definition}\label{def_1}
		We say that
		\begin{itemize}
			\item a pair $(u,p)\in H^1(\Omega_h)\times L^2(\Omega_h)$ is a weak solution to \eqref{eq_main} if $u$ satisfies the boundary condition \eqref{eq_main}$_2$ in the trace sense and for every $\psi\in H_0^1(\Omega_h)$ we have
			\begin{equation}
				\eta\int_{\Omega_h}\nabla u:\nabla\psi-\int_{\Omega_h}p\nabla\cdot\psi+\int_{\Omega_h}u\cdot\nabla u\cdot\psi=0;
			\end{equation}
			\item a pair $(u,p)\in W^{2,2+\sigma}(\Omega_h)\times W^{1,2+\sigma}(\Omega_h)$, for $\sigma$ as in \eqref{eq_sigma}, is a strong solution to \eqref{eq_main} if $u$ satisfies the boundary condition \eqref{eq_main}$_2$ pointwise and the equation \eqref{eq_main}$_1$ a.e. in $\Omega_h$.
		\end{itemize}
	\end{definition}
We show that, under suitable assumptions, \eqref{eq_main} admits a strong solution.
	
\begin{theorem}\label{the_ex_regu}Let $h\in C_0^{1,1}([-1,1])$ be small enough to ensure \eqref{eq_bound_infty}. For any $\gamma\geq0$, \eqref{eq_main} admits a weak solution $(u, p)$ in $\Omega_h$; there exists $\overline{\gamma}>0$ such that the weak solution is unique for $\gamma\in [0, \overline{\gamma})$. Moreover, any weak solution is a strong solution and there exists $C^*>0$, depending on $\sigma$ but not on $h$, such that
\begin{equation}\label{reg-est}
\left\|u\right\|_{W^{2,2+\sigma }(\Omega_h)}+\left\|p\right\|_{W^{1,2+\sigma}(\Omega_h)}\le C^*\gamma \qquad \forall  \gamma\in[0, \overline{\gamma}).
\end{equation}
\end{theorem}
\begin{proof}
We preliminarily observe that $\partial\mathcal{B}^h=\mathcal{B}_{-1}\cup \mathcal{B}_1\cup\mathcal{S}^h$, where $\mathcal{S}^h$ is the translation of $\mathcal{S}$ through the function $h$ and, hence, it is of class $C^{1,1}$.

\paragraph{\emph{Existence and uniqueness of weak solutions}} Since $\Omega_h$ is a Lipschitz domain with $\partial\Omega_h$ having only one connected component and since
\begin{equation*}
\int_{\partial\Omega_h}u_*\cdot n =0,
\end{equation*}
by \cite{LE1,Hopf1941} for every $\varepsilon>0$ there exists a Leray-Hopf solenoidal extension $s\in H^1(\Omega_h)$ of the boundary datum $u_*$ satisfying
\begin{equation*}
-\int_{\Omega_h}(v\cdot\nabla s)v\le\varepsilon\left\|\nabla v\right\|_{L^2(\Omega_h)}^2\hspace{.3in}\forall v\in H_0^1(\Omega_h).
\end{equation*}
Letting $v=u-s$, we recast \eqref{eq_main} into
\begin{equation}\label{eq_main2}
\begin{aligned}
&-\eta\Delta v+\nabla p+v\cdot\nabla v=f,&\qquad&\nabla\cdot v=0&\qquad&\mbox{in}\hspace{.1in}\Omega_h\\
&v=0&\qquad&&\qquad&\mbox{on}\hspace{.1in}\partial\Omega_h
\end{aligned}
\end{equation}
where $f=\eta\Delta s-s\cdot\nabla v-v\cdot\nabla s-s\cdot\nabla s$. A definition of weak and strong solution to \eqref{eq_main2} follows from Definition \ref{def_1}, the only difference being that the right-hand side is now given by the $L^2$-inner product between $f$ and every $\psi\in H_0^1(\Omega_h)$. Let $V(\Omega_h)=\left\{\varphi\in H_0^1(\Omega_h):\,\nabla\cdot\varphi=0\text{ a.e. in }\Omega_h\right\}$. From \cite[Lemma IX.1.2]{GA1} it follows that $v$ is a weak solution to \eqref{eq_main2} if and only if for every $\varphi\in V(\Omega_h)$ we have
\begin{equation}\label{eq_12}
\eta\int_{\Omega_h}\nabla v:\nabla\varphi+\int_{\Omega_h}v\cdot\nabla v\cdot\varphi=\left\langle f,\varphi\right\rangle.
\end{equation}
where $\left\langle\cdot,\cdot\right\rangle$ represents the duality paring in $H^{-1}(\Omega_h)$. From \cite[Theorem IX.4.1]{GA1} we know that there exists $v\in V(\Omega_h)$ satisfying \eqref{eq_12} which is equivalent to say that $u=v+s\in H^1(\Omega_h)$, together with the associated pressure $p\in L^2(\Omega_h)$ (which is chosen with zero mean over $\Omega_h$), is a weak solution to \eqref{eq_main}. In addition, from \cite[Theorem IX.4.1]{GA1} it follows that
\begin{equation}\label{eq_21}
\left\|u\right\|_{H^1(\Omega_h)}\le C\left[\left(1+\frac{1}{\eta}\right)\left\|s\right\|_{H^1(\Omega_h)}+\frac{1}{\eta}\left\|s\right\|_{H^1(\Omega_h)}^2\right]\le C\left(\gamma+\gamma^2\right).
\end{equation}
where $C>0$ does not depend on $h$ due to the smallness assumption on  its $C^{1,1}$-norm.

Next, we prove that the weak solution $u\in H^1(\Omega_h)$ constructed above is unique for sufficiently small $\gamma\ge0$. Let $u_1,u_2\in H^1(\Omega_h)$ be any two weak solutions to \eqref{eq_main}, and define $w=u_1-u_2$.
Then $w$ solves
\begin{equation}\label{eq_20}
\eta\int_{\Omega_h}\nabla w:\nabla\varphi+\int_{\Omega_h}w\cdot\nabla w\cdot\varphi=-\int_{\Omega_h}w\cdot\nabla u_2\cdot\varphi+u_2\cdot\nabla w\cdot\varphi
\end{equation}
for every $\varphi\in V(\Omega_h)$. Hence, choosing $\varphi=w$ in \eqref{eq_20}, and using H\"older and Sobolev inequality, we get
\begin{align}\label{eq_22}
\eta\left\|\nabla w\right\|_{L^2(\Omega_h)}^2&=-\int_{\Omega_h}w\cdot\nabla u_2\cdot w\le\left\|\nabla u_2\right\|_{L^2(\Omega_h)}\left\|w\right\|_{L^4(\Omega_h)}^2 \notag\\
&\le C\left\|\nabla u_2\right\|_{L^2(\Omega_h)}\left\|\nabla w\right\|_{L^2(\Omega_h)}^2.
\end{align}
Since $u_2$ is a weak solution to \eqref{eq_main}, injecting \eqref{eq_21} into \eqref{eq_22}, we get
\begin{equation}\label{eq_23}
\eta\left\|\nabla w\right\|_{L^2(\Omega_h)}^2\le C\left(\gamma+\gamma^2\right)\left\|\nabla w\right\|_{L^2(\Omega_h)}^2.
\end{equation}
Taking $\gamma>0$ small enough such that
\begin{equation*}
C\left(\gamma+\gamma^2\right)<\eta
\end{equation*}
from \eqref{eq_23} we arrive at $\left\|\nabla w\right\|_{L^2(\Omega_h)}=0$, and since $w=0$ on $\partial\Omega_h$ we conclude that $w\equiv0$, i.e., $u_1=u_2$ in $\Omega_h$.

\paragraph{\emph{Regularity}}

The boundary $\partial\Omega_h$ is locally of class $C^{1,1}$ except at the vertices and the edges, including the closed curves $\partial \mathcal{B}_{\pm 1}^h$. Therefore, we need
to analyze only the behavior of a solution $(u,p)$ close to any of the three just mentioned kinds of boundary points.\par
We start with the analysis of the regularity of $(u,p)$ close to the points belonging to the closed edges of $P$, thereby including the eight
vertices.
Consider the parallelepiped $P^l=(-R,R)\times[1,3)\times(-1,1)$ and the reflection $\mathcal{B}^{h,l}$ of $\mathcal{B}^h$ with respect to the
plane $y=1$, and let $\Omega_h^l=P^l\setminus\mathcal{B}^{h,l}$. If we define $u^l:\Omega_h^l\to\mathbb{R}^3$ and $p^l:\Omega_h^l\to\mathbb{R}$ as
\begin{align*}
u^l(x,y,z)&=\big(u^l_1(x,y,z),u^l_2(x,y,z),u^l_3(x,y,z)\big)\\
&=\big(u_1(x,2-y,z),-u_2(x,2-y,z),u_3(x,2-y,z)\big)\\
p^l(x,y,z)&=p(x,2-y,z),
\end{align*}
then $(u^l,p^l)$ solves
\begin{align*}
-\eta\Delta u^l+\nabla p^l+u^l\cdot\nabla u^l=0,\hspace{.3in}\nabla\cdot u^l=0&\hspace{.3in}\mbox{in}\hspace{.1in}\Omega_h^l.
\end{align*}
Hence, the pair $(\overline{u},\overline{p})$ defined as
\begin{equation*}
(\overline{u},\overline{p})=
\begin{cases}
(u,p)&\hspace{.3in}\mbox{in}\hspace{.1in}\Omega_h\\
(u^l,p^l)&\hspace{.3in}\mbox{in}\hspace{.1in}\Omega_h^l
\end{cases}
\end{equation*}
is smooth across $\Gamma^+_y$ except at $\partial\mathcal{B}^h_1$ and solves
\begin{equation*}
-\eta\Delta\overline{u}+\nabla\overline{p}+\overline{u}\cdot\nabla\overline{u}=0,\hspace{.3in}\nabla\cdot\overline{u}=0\hspace{.3in}
\mbox{in}\hspace{.1in}\Omega_h\cup\Omega_h^l.
\end{equation*}
Thanks to this reflection, all the points of $\Gamma^+_y\setminus\partial\mathcal{B}^h_1$ are interior points of $\Omega_h\cup\Omega_h^l$.
In a similar manner, we introduce reflections with respect to the planes $y=-1$ and $z=-1$ over which homogeneous Dirichlet boundary
conditions for $u$ are satisfied. We then find a pair $(\tilde{u},\tilde{p})$, defined in $\tilde{\Omega}_h=(-R,R)\times(-3,3)\times(-3,1)\setminus\{\mathcal{B}^h \mbox{ and its reflections}\}$, solving
\begin{equation*}
-\eta\Delta\tilde{u}+\nabla\tilde{p}+\tilde{u}\cdot\nabla\tilde{u}=0,\hspace{.3in}\nabla\cdot\tilde{u}=0\hspace{.3in}\mbox{in}\hspace{.1in}\tilde{\Omega}_h
\end{equation*}
complemented with reflected boundary conditions that we do not need to make fully explicit.
We then reflect $\tilde{\Omega}_h$ with respect to the plane
$z=1$, which is slightly more delicate because of the inhomogeneous Dirichlet condition. However, we have $u_3=0$ on $\Gamma^+_z$ which allows an odd reflection for this (normal) component; the tangential components $u_1$ and $u_2$ may be evenly reflected without altering their boundary regularity.
Defining the so-obtained domain $\hat{\Omega}_h\supsetneqq \tilde{\Omega}_h$, the points in the edges of $\partial P$ belong either to the interior of $\hat{\Omega}_h$ or to flat regions of $\partial \hat{\Omega}_h$, where standard elliptic regularity theory applies (interior or up to the boundary, see \cite{GA1}).
Hence,
$(u,p)\in W^{2,2+\sigma}(U)\times W^{1,2+\sigma}(U)$ in a neighborhood $U$ of any edge of $\partial P$ and, in view of \eqref{u*}, the following estimate holds
\begin{equation}\label{reg-est-vertices}
\left\|u\right\|_{W^{2,2+\sigma}(U)}+\left\|p\right\|_{W^{1,2+\sigma}(U)}\le C\left\|u_*\right\|_{W^{2-\frac{1}{2+\sigma},2+\sigma}(\partial\Omega_h)}\leq C \lambda
\end{equation}for some $C>0$ independent of $h$. We point out that, due to the absence of an odd reflection of $u_1$ (the normal velocity component), the reflection technique {\em does not} apply
to the faces $x=\pm R$.\par
We then focus on the boundary points of the planar curve $\partial\mathcal{B}_1^h$ (see Figure \ref{angle}), and one could argue analogously for the points of $\partial\mathcal{B}_{-1}^h$. The analysis can be performed using the theory developed in \cite{Mazya-Rossmann2010} but, since we were unable to find therein the exact statement that we need, we briefly sketch how to combine the tools of this theory to reach regularity and related
bounds.\par
Let $\mathcal{G}\subset \Omega_h$ be a neighborhood of $\partial\mathcal{B}_1^h$ with a smooth boundary far away from it.
Consider a smooth cut-off function  $\chi:\Omega_h \to [0,1]$ with ${\rm supp}(\chi)$ strictly contained in $\mathcal{G}$. Then,
$(\underline{u}, \underline{p})=(\chi u, \chi p) \in H^{1}(\mathcal{G})\times L^{2}(\mathcal{G})$ satisfies

\begin{equation}\label{eq_30}
	\begin{aligned}
		&-\eta\Delta \underline{u}+\underline{u}\cdot\nabla\underline{u}+\nabla \underline{p}=\underline{f},&\qquad&\nabla\cdot \underline{u} =\underline{g}&\qquad&\mbox{in}\hspace{.1in}\mathcal{G}\\[3pt]
		&\underline{u}=0&\qquad&&\qquad&\mbox{on}\hspace{.1in}\partial \mathcal{G}
	\end{aligned}
\end{equation}
with \begin{equation*}
\underline{f} =-\eta u\Delta\chi -2 \eta\nabla\chi\cdot \nabla u + p\nabla\chi  -\chi u\cdot\nabla u +\chi^2u\cdot\nabla u+\chi u(u\cdot\nabla\chi), \quad \underline{g}=u\cdot\nabla \chi.
\end{equation*}

Since  $\mathcal{G}$ is a polyhedral-type domain in the sense of \cite[Section 4.1.1]{Mazya-Rossmann2010} and $(\underline{f},\underline{g})\in L^{3/2}(\mathcal{G})\times H^1(\mathcal{G})$, we fall in the geometric and functional setting of \cite[Theorem 11.2.8]{Mazya-Rossmann2010}. For any  $A\in\partial\mathcal{B}_1^h$, let $\mathcal{D}_A$ be the dihedron whose faces are respectively contained in the plane $\{y=1\}$ and the plane $\Lambda$ perpendicular to $\mathcal{P}_\alpha$ and tangent to $\partial\mathcal{B}^h$ at $A$. The regularity theory for \eqref{eq_30} is based on the study of the \emph{operator pencil} $\mathcal{A}_A$ associated with the model problem in $\mathcal{D}_A$ for any $A\in \partial \mathcal{B}_1^h$;
 see \cite[Section 2.3.1]{Mazya-Rossmann2010}. Roughly speaking, $\mathcal{A}_A$ describes the behavior of the elliptic operator related to \eqref{eq_30} around $A$.  Let
  $\mu_+(A)$ be the largest positive number such that the strip
 \begin{equation*}
0<{\rm Re}\ \lambda<\mu_+(A)
\end{equation*}
contains only the eigenvalue $\lambda=1$ of $\mathcal{A}_A$.  According to \cite[Section 11.1.2]{Mazya-Rossmann2010}, in the case of Dirichlet boundary conditions, we have $\mu_+(A)= \pi/\theta(A)$. Moreover, due to \eqref{eq_theta_*},
\begin{equation}\label{mu+}
\overline{\mu}_+=\min_{A\in\partial\mathcal{B}_1^h}\mu_+(A)\geq \frac{\pi}{\theta_*}>1.
\end{equation}
 Then, \cite[Theorem 11.2.8]{Mazya-Rossmann2010} holds with $\beta=0$ and $\delta=0$ whenever
\begin{equation}\label{eq_1}
\max\left\{0,2-\overline{\mu}_+\right\}<\frac{2}{s}< 2.
\end{equation}
 We apply it twice: first with $s=3/2$, yielding $(\underline{u},\underline{p})\in W^{2,3/2}(\mathcal{G})\times W^{1, 3/2}(\mathcal{G})$ and, thereby,  $(\underline{f},\underline{g})\in L^{2+\sigma}(\mathcal{G})\times W^{1, 2+\sigma}(\mathcal{G})$; second with $s=2+\sigma$ for $\sigma$ satisfying \eqref{eq_sigma}, concluding that $(\underline{u},\underline{p})\in W^{2,2+\sigma}(\mathcal{G})\times W^{1,2+\sigma}(\mathcal{G})$. We stress that this result holds in the mentioned non-weighted spaces after employing suitable embeddings for $s<3$. \par
Since the corresponding regularity estimate does not appear explicitly in \cite[Theorem 11.2.8]{Mazya-Rossmann2010}, we describe here the main steps to achieve it.

We fix $A\in \partial\mathcal{B}_1^h$ and localize \eqref{eq_30} in a neighborhood $\mathcal{U}\subset \mathcal{G}$ of $A$ by introducing another cut-off function. We consider a suitable $C^{1,1}$ diffeomorphism $\kappa: \mathcal{D}_A\rightarrow \kappa(\mathcal{D}_A)\subset \mathbb{R}^3$ such that  $\kappa^{-1}(\mathcal{U})\subset\kappa^{-1}(\mathcal{G})$, $\kappa^{-1}(\Gamma_{y}^+\cap \mathcal{U}) \subset\{y=1\}$ and $\kappa^{-1}(\partial\mathcal{B}^h \cap  \mathcal{U})\subset\Lambda$, see \cite[Proposition 6.2]{BocCasGan} for the details in the 2D setting.
Then, the pair $(\hat{u}, \hat{p})=(\underline{u}, \underline{p})\circ \kappa \in W^{2, 2+\sigma}(\kappa^{-1}(\mathcal{U}))\times W^{1, 2+\sigma}(\kappa^{-1}(\mathcal{U}))$ solves a transformed Navier-Stokes problem (similar to \eqref{BVprob_trans} below) that can be reduced to the Stokes-type problem
\begin{equation}\label{eq_3}
\begin{aligned}
&-\eta\Delta \hat{u}+\nabla \hat{p}=\hat{f},&\qquad&\nabla\cdot \hat{u}=\hat{g}&\qquad&\mbox{in}\hspace{.1in}\kappa^{-1}(\mathcal{U})\\
&\hat{u}=0&\qquad&&\qquad&\mbox{on}\hspace{.1in}\partial\kappa^{-1}(\mathcal{U})
\end{aligned}
\end{equation}
with a suitable $\hat{f}\in L^{2+\sigma}(\kappa^{-1}(\mathcal{U}))$ and $\hat{g}\in W^{1,2+\sigma}(\kappa^{-1}(\mathcal{U}))$. We artificially construct three polyhedral cones with vertex $A$ as follows. Take the line $r_A$ containing $A$ and orthogonal to the plane $\{y=1\}$. The plane $\mathcal{P}_\alpha$ in \eqref{Palpha} contains $r_A$ and hence is also orthogonal to $\{y=1\}$. Consider the planes  $\mathcal{P}_\alpha^\pm$ obtained by rotating $\mathcal{P}_\alpha$ around $r_A$ of angles of measures $\pm \pi/4$. These planes allow us to construct three polyhedral cones: the first two delimited by $\{y=1\}$, $\Lambda$, $\mathcal{P}_\alpha$ and the third by $\{y=1\}$ and $\mathcal{P}_\alpha^\pm$.
Through a smooth partition of unity we obtain three Stokes-type problems with homogeneous Dirichlet conditions as \eqref{eq_3}, one in each polyhedral cone. \par Our next goal is to apply \cite[Theorem 10.5.4]{Mazya-Rossmann2010} to these problems. In this scenario, the regularity theory is based on the study of the operator pencils related to the vertex $A$ and related to the edges $M_k$ ($k=1,2,3$) of each cone.
On the one hand, since each cone is contained in a half-space, \cite[Section 11.3.1]{Mazya-Rossmann2010} guarantees that the hypothesis of the theorem on the eigenvalues of the operator pencils at $A$ holds for $s=2+\sigma$ and $\beta=0$.
On the other hand, the construction of the cones ensures that the $\overline{\mu}^k_+$'s, defined as in \eqref{mu+} and corresponding to the $M_k$'s satisfy \eqref{eq_1} with $s=2+\sigma$. Through a further localizing cut-off in $\mathcal{D}_A$, we may then  apply \cite[Theorem 10.5.4]{Mazya-Rossmann2010} with $s=2+\sigma$, $\beta=0$, $\delta=0 $ and, thanks to the above partition of unity,  we obtain
\begin{equation*}
\left\|\hat{u}\right\|_{W^{2,2+\sigma}(\kappa^{-1}(\mathcal{U}))}+\left\|\hat{p}\right\|_{W^{1,2+\sigma}(\kappa^{-1}(\mathcal{U}))}\le C\left(\|\hat{f}\|_{L^{2+\sigma}(\kappa^{-1}(\mathcal{U}))}+\left\|\hat{g}\right\|_{W^{1,2+\sigma}(\kappa^{-1}(\mathcal{U}))}\right)
\end{equation*}
which, in turn, by letting $A$ vary and by compactness of $\partial\mathcal{B}_1^h$, implies
\begin{equation*}
\left\|\underline{u}\right\|_{W^{2,2+\sigma}(\mathcal{G})}+\left\|\underline{p}\right\|_{W^{1,2+\sigma}(\mathcal{G})}\le C\left(\|\underline{f}\|_{L^{2+\sigma}(\mathcal{G})}+\left\|\underline{g}\right\|_{W^{1,2+\sigma}(\mathcal{G})}\right)\leq C \gamma  \quad \forall  \gamma\in[0, \overline{\gamma}),
\end{equation*} where the last inequality is obtained after bootstrapping and by \eqref{u*}. Similarly, we get the analogous bound in a neighborhood of $\partial\mathcal{B}_{-1}^h$. Finally, after collecting terms, we combine the obtained estimates with \eqref{reg-est-vertices} and the bounds up to $\mathcal{S}^h$ to infer that there exists $C^*>0$, independent of $h$ (due to the smallness assumption on its $C^{1,1}$-norm), such that
\begin{equation*}
\left\|u\right\|_{W^{2,2+\sigma }(\Omega_h)}+\left\|p\right\|_{W^{1,2+\sigma}(\Omega_h)}\le C^*\gamma \qquad \forall  \gamma\in[0, \overline{\gamma}).
\end{equation*}
The proof is so complete.
\end{proof}

\section{Fluid-structure interaction problem}\label{sec_FSI}

Aiming to study the equilibrium configurations for the FSI problem \eqref{BVprobRR} and \eqref{compa-beam}, we let the domain $\mathcal{B}^h$ vary.
     For any $y\in(-1,1)$, the cross-section $\mathcal{B}_y^h$ has a barycenter with vertical displacement $h(y)$, which is now unknown.  \par
     The FSI problem we are led to study consists in the coupled PDE-ODE system
     \begin{equation}\label{BVprobR}
		\begin{aligned}
			&-\eta\Delta u+u\cdot\nabla u+\nabla p=0,\quad\nabla\cdot u=0 \quad \text{in}\quad \Omega_h=P\setminus \mathcal{B}^h\\[2pt]
			&u=0 \quad \text{on} \quad  \partial \mathcal{B}^h ,\quad u=\gamma Ve_1\quad \text{on} \quad \overline{\Gamma_{y}^\pm \setminus \mathcal{B}_{\pm 1}}\cup\Gamma_z^\pm\\[2pt]
			&u= \gamma V_ie_1\quad \text{on} \quad \Gamma_x^-,\quad u= \gamma V_oe_1\quad\text{on} \quad \Gamma_x^+.
		\end{aligned}
	\end{equation}

    \begin{equation}\label{beam-pb}
		\begin{aligned}
			&h''''+ f(h)= L(\cdot, h)  \qquad \text{in} \quad (-1,1), \\[2pt]
			&h(\pm 1) = 0 , \qquad h'(\pm 1)=0 \ \text{ or } \ h''(\pm 1)=0,
		\end{aligned}
	\end{equation} where
    \begin{equation}\label{load}
        L(y, h)=L(y, u(h), p(h)) = - e_3 \cdot \int_{\partial \mathcal{B}_y^h}\mathbb{T}(u(h), p(h)) n
    \end{equation} with $\mathbb{T}(\cdot, \cdot)$ as in \eqref{lift}.
Let us state the main result on the equilibrium configurations of the fluid-beam interaction problem \eqref{BVprobR}-\eqref{beam-pb}. To this end, we introduce the open ball \begin{equation}\label{ballH}\mathbb{B}_\delta= \{ h \in \mathbb{H}^4 \ \text{s.t.} \ \|h\|_{\mathbb{H}^4} < \delta\}\\[5pt]\end{equation} and we prove the following:
\begin{theorem}\label{main-theo}
There exists $\beta_0>0$ and $\gamma_0\in (0, \overline{\gamma}]$, with $\overline{\gamma}$ as in Theorem \ref{the_ex_regu},  such that for any $\gamma \in [0, \gamma_0)$, the FSI problem \eqref{BVprobR}-\eqref{beam-pb} admits a unique solution $$(h, u, p)\in \mathbb{B}_{\beta_0}\times W^{2, 2+\sigma}(\Omega_h)\times W^{1, 2+\sigma}(\Omega_h).$$In addition, there exists $\gamma^*\in (0, \gamma_0]$ such that the map $\gamma\mapsto (h(\gamma), u(\gamma), p(\gamma))$ is Lipschitz continuous on the interval
$[0, \gamma^*)$. More precisely, for any $\gamma_1, \gamma_2 \in [0, \gamma^*)$,
\begin{equation*}\begin{aligned}
    \|h(\gamma_1) - h(\gamma_2)\|&_{\mathbb{H}^4} + \|u(\gamma_1) - u(\gamma_2)\circ \varphi_{h(\gamma_2)-h(\gamma_1)}\|_{W^{2, 2+\sigma}(\Omega_{h(\gamma_1)})} \\[5pt]&+ \|p(\gamma_1) - p(\gamma_2)\circ \varphi_{h(\gamma_2)-h(\gamma_1)}\|_{W^{1, 2+\sigma}(\Omega_{h(\gamma_1)})}\leq C(\beta_0) |\gamma_1-\gamma_2|,
    \end{aligned}
\end{equation*}
where the diffeomorphism $\varphi_h$ is defined in Lemma \ref{lemma-diffeo} below.

\end{theorem}\smallskip

The proof of Theorem \ref{main-theo}, given in Section \ref{sec-proof}, is based on a fixed-point procedure that involves several steps:
\begin{itemize}
\item we derive technical results about the lift $L$;
\item for any $\overline{h}\in C^{1,1}_0([-1,1])$, we apply Theorem \ref{the_ex_regu} to obtain   $(u(\overline{h}), p(\overline{h}))$ solving \eqref{BVprobRR} in $\Omega_{\overline{h}}$;

\item we reduce \eqref{beam-pb} to
a related beam equation with a given source depending on $\overline{h}$;
\item we prove existence and uniqueness of this new equation thereby providing a solution $\mathfrak{h}(\overline{h})\in C^{1,1}_0([-1,1])$;
\item we show that $\overline{h}\mapsto \mathfrak{h}(\overline{h})$ admits a unique fixed point in a suitable space.
\end{itemize}

\subsection{Preliminary results}
We first prove two properties of the lift force.

\begin{lemma}\label{lem-Lbound}
Let $h\in C^{1,1}_0([-1,1])$ satisfy the smallness assumption in Theorem \ref{the_ex_regu}. Let $\gamma \in [0, \overline{\gamma})$ and $(u,p)\in W^{2,2+\sigma}(\Omega_h)\times W^{1,2+\sigma}(\Omega_h)$ be the unique solution to \eqref{eq_main} from Theorem \ref{the_ex_regu}. Then, there exists  $C>0$, independent of $h$,  such that
\begin{equation}\label{load-control}
\|L(\cdot, h)\|_{L^\infty(-1,1)}\leq C \gamma .
\end{equation}
Moreover, $L(\cdot, h)\in C(-1,1)$.
\end{lemma}
\begin{proof}
Since $u\in W^{2,2+\sigma}(\Omega_h)$, we have that $\nabla u\in W^{1,2+\sigma}(\Omega_h)$ and it follows from the Trace Theorem  \cite[Theorem 3.37]{McLean2000}  that $\nabla u\in W^{1-\frac{1}{2+\sigma},2+\sigma}(\partial \mathcal{B}^h)$. The boundedness of $\partial \mathcal{B}^h$ and  H\"older inequality then imply that $\nabla u\in H^{1-\frac{1}{2+\sigma}}(\partial \mathcal{B}^h)$. Since $\sigma >0$, we can apply the Trace Theorem once again to get $\nabla u\in H^{\frac{1}{2}-\frac{1}{2+\sigma}}(\partial \mathcal{B}^h_y)$ and, in particular, $\nabla u\in L^2(\partial \mathcal{B}^h_y)$ for a.e.\ $y\in(-1,1)$. Moreover, the following chain of inequalities holds:
\begin{align*}
\left\|\nabla u\right\|_{L^1(\partial \mathcal{B}^h_y)}&\le C\left\|\nabla u\right\|_{L^2(\partial \mathcal{B}^h_y)}\le C\left\|\nabla u\right\|_{H^{\frac{1}{2}-\frac{1}{2+\sigma}}(\partial \mathcal{B}^h_y)}\\
&\le C\left\|\nabla u\right\|_{W^{1-\frac{1}{2+\sigma},2+\sigma}(\partial \mathcal{B}^h)}\le C\left\|u\right\|_{W^{2,2+\sigma}(\Omega_h)},
\end{align*}where the constants $C>0$ may differ but are all independent of $y$.
Similarly, since $p\in W^{1,2+\sigma}(\Omega_h)$ we obtain that $p\in H^{\frac{1}{2}-\frac{1}{2+\sigma}}(\partial \mathcal{B}^h_y)$ and, in particular, $p\in L^2(\partial \mathcal{B}^h_y)$. Thus, we end up with the following inequality
\begin{equation}
\left\|p\right\|_{L^1(\partial \mathcal{B}^h_y)}\le C\left\|p\right\|_{W^{1,2+\sigma}(\Omega_h)}.
\end{equation}
Therefore, we infer from the definition \eqref{load} that
\begin{equation*}
|L(y,h)|\le C\left(\left\|u\right\|_{W^{2,2+\sigma}(\Omega_h)}+\left\|p\right\|_{W^{1,1+\sigma}(\Omega_h)}\right)
\end{equation*}
for a.e. $y\in(-1,1)$ and, using the uniform regularity estimate \eqref{reg-est}, we find \eqref{load-control}.\par
As already noticed in the proof of Theorem \ref{the_ex_regu}, $\mathcal{S}^h$ is of class $C^{1,1}$. Since $u=0$ on $\mathcal{S}^h$, by
elliptic regularity, we then infer that $u\in W^{2,q}(\mathcal{N})$ in any (closed) neighbourhood $\mathcal{N}\subset\overline{\Omega}_h$ of any
point $\xi\in\mathcal{S}^h$. By Embedding Theorems, this implies $u\in C^{1,\alpha}(\mathcal{N})$ for any $\alpha<1$ and, hence,
$\nabla u\in C^{0,\alpha}(\mathcal{N})$. By using \eqref{eq_main}, we then infer that $p\in C^{0,\alpha}(\mathcal{N})$ for any $\alpha<1$.
Therefore, the integrand in \eqref{load} is continuous for any $y\in(-1,1)$ and, by letting $y$ vary, $L(\cdot, h)\in C(-1,1)$.
\end{proof}

\begin{remark}\label{cont-lift}
The fact that $L(\cdot, h)\in C(-1,1)$ validates the pointwise definition \eqref{load} of the lift.
The continuity over $[-1,1]$ could be achieved if the solution to \eqref{BVprobR} were more regular, but within weighted spaces (see for instance
\cite[Theorem 11.1.5]{Mazya-Rossmann2010}), or if $\theta(A)<\pi/2$ for any $A$, see Figure \ref{angle}, but this would not physically describe an actual bridge.
\end{remark}

Thanks to Theorem \ref{the_ex_regu},  we next discuss the well-posedness of \eqref{beam-pb} for the fluid velocity and pressure $ (u(\overline{h}), p(\overline{h}))$ determined by a given profile $\overline{h}$. We consider the modified ODE problem
\begin{equation}\label{mod-beam-pb}
		\begin{aligned}
			&h''''+ f(h)= L(\cdot, \overline{h})  \qquad \text{in} \quad (-1,1) \\[2pt]
			&h(\pm 1) = 0 , \qquad h'(\pm 1)=0 \ \text{ or } \ h''(\pm 1)=0.
		\end{aligned}
	\end{equation}

We introduce the nonlinear map that will have a fixed point.
\begin{proposition}\label{prop-exi-beam}
Let $\overline{h}\in C^{1,1}_0([-1,1]) $ satisfy the smallness assumption  in Theorem \ref{the_ex_regu} and let $\gamma \in [0, \overline{\gamma})$ . Then, \eqref{mod-beam-pb} admits a unique solution  $h\in \mathbb{H}^4\subset C^{1,1}_0([-1,1])$ and there exists $C>0$, independent of $\overline{h}$,  such that
\begin{equation}\label{bound-beam}
    \| h \|_{\mathbb{H}^4}\leq C\gamma
\end{equation}
\end{proposition}
\begin{proof}Since $\overline{h}$ is given, we may rewrite \eqref{mod-beam-pb} as
\begin{equation}\label{compatibility2}
h''''(y)+f\big(h(y)\big)=g(y)= L(y, \overline{h}(y))\qquad \text{for a.e. } y\in(-1,1)\, ,
\end{equation}where $g\in L^\infty(-1,1)$ is a given function of $y$. Weak solutions to \eqref{compatibility2} are critical points  of the related energy functional
\begin{equation}\label{energy-fun}
 h\mapsto E(h)=\int_{-1}^{1}\left(\frac{h''(y)^{2}}2 +F\big(h(y)\big) - g(y)h(y)\, \right)dy\,.
\end{equation} Due to \eqref{assf}, the integrand in $E$ is strongly convex with respect to $(h, h'')\in \mathbb{R}^2$ and hence $E$ is strictly convex with respect to $h$ on $\mathbb{H}$ (also if $F(h)\equiv f(h)\equiv0$), so that $E$ admits a unique minimizer in $\mathbb{H}$ that weakly solves \eqref{compatibility2}.
 Moreover, taking the solution itself as a test function in the weak formulation, we obtain
\begin{equation}
    \|h\|^2_{\mathbb{H}} + \int_{-1}^1 f(h)h  =\int_{-1}^1 gh
\end{equation}
and, by \eqref{assf} and \eqref{SH-embedding}, we get the bound
    \begin{equation}\label{bound-hH}
       \|h\|_{\mathbb{H}} \leq 2S_{\mathbb{H}} \|g\|_{L^{\infty}(-1,1)}.
    \end{equation}
   The solution $h$ is, in fact, more regular. Indeed,  from \eqref{assf}, we know that $f(h)\in L^\infty(-1,1)\subset L^2(-1,1)$ and, from Lemma \ref{lem-Lbound}, that $g\in L^\infty(-1,1)\subset L^2(-1,1)$. Then, by elliptic regularity, it follows that $h\in \mathbb{H}^4 $. Moreover, there exist constants $C>0$ (which may vary in each inequality) such that
\begin{equation*}\begin{aligned}
\|h\|_{\mathbb{H}^4}& \leq C \left(\|f(h)\|_{L^2(-1,1)} + \|g\|_{L^2(-1,1)}\right) \\[5pt]&\leq C\left( \|h\|_{L^\infty(-1,1)} + \|g\|_{L^\infty(-1,1)}\right)\stackrel{\eqref{bound-hH}}\leq   C\|g\|_{L^\infty(-1,1)}\leq C \gamma
\end{aligned}
\end{equation*} where the last inequality is due to Lemma \ref{lem-Lbound} applied to $L(\cdot, \overline{h})=g(\cdot)$.
\end{proof}

The proof of Proposition \ref{prop-exi-beam}  shows that the solution to \eqref{compatibility2}  belongs to $W^{4, \infty} (-1,1)$, while Lemma \ref{lem-Lbound} implies that it also belongs to $C^4(-1,1)$. However, this additional regularity will not be used in the sequel.\par

In order to study the dependence of the lift on the beam profile $h$, it is convenient to introduce a setting where the fluid domain is independent of $h$. To this end, we define a suitable map $\varphi_h$.
\begin{lemma}\label{lemma-diffeo}
Let $\mathcal{D}\subset \Gamma_y^-$ be an open set such that the cylinder $\mathcal{C}=\mathcal{D} \times [-1,1]$ strictly contains any $\mathcal{B}^h$ for $h$ sufficiently small. Consider a smooth cut-off function $\xi:[-R, R]\times [-1,1]\rightarrow [0,1]$ such that
\begin{equation}\label{cutoff-diffeo}\begin{aligned}\xi=1\quad  \text{in} \quad \overline{\mathcal{D}} \quad \text{and} \quad
\xi=0 \quad \text{in}\quad \left([-R\times R]\times [-1,1]\right)\setminus \lambda\mathcal{D}
    \end{aligned}
\end{equation}with $\lambda>1$ ensuring that the homothetic $\lambda \mathcal{C}\subset \Gamma_y^-$.
For $h\in \mathbb{H}$, consider the map
	\begin{equation}\label{diffeo}\varphi_h (x,y,z)= \left(x, y, z+ \xi(x,z)h(y)\right).\end{equation} Then, there exists $\beta>0$ such that, for $\|h\|_{\mathbb{H}}<\beta$,  $\varphi_h$ is a $C^1$-diffeomorphism from $\overline{\Omega}_0$ to $\overline{\Omega}_h$ satisfying \begin{equation}\label{diff-prop}
			\varphi_h (\partial\mathcal{B})=\partial\mathcal{B}_h , \quad \varphi_h (\overline{\Gamma_x^\pm })=\overline{\Gamma_x^\pm},
           \quad \varphi_h (\overline{\Gamma_y^\pm \setminus \mathcal{B}_{\pm 1} })=\overline{\Gamma_y^\pm \setminus \mathcal{B}_{\pm 1} } \ \
              \text{and} \ \ \varphi_h (\overline{\Gamma_z^\pm })=\overline{\Gamma_z^\pm}.
            \end{equation}
\end{lemma}
\begin{proof}First, we take $\beta$  small enough such that \eqref{eq_bound_infty} holds, which is possible in view of \eqref{SH-embedding}. The embedding $\mathbb{H}\subset C^1([-1,1])$ and the smoothness of $\xi$ yield $C^1$-regularity of $\varphi_h$. Next, we want to show that the determinant of its Jacobian
$$J_h(x,y,z)= \left(\begin{matrix}
1&0&0\\[5pt]
0&1&0\\[5pt]
\partial_x\xi(x,z)h(y)&\xi(x,z)h'(y)&1+ \partial_z\xi(x,z)h(y)
\end{matrix}\right)$$
is positive and uniformly bounded in $\Omega_0$, that is, there exists $C>0$ such that
\begin{equation}\label{detJeta} C \leq \det(J_h)=1 + \partial_z\xi(x,z) h(y) \leq \frac{1}{C}, \qquad (x,y,z).\in \Omega_0\end{equation}
It follows from \eqref{SH-embedding} that
\begin{equation}\label{est-detJeta}
 1- C_1\|h\|_{\mathbb{H}}\leq \det(J_h)\leq  1 +C_2 \|h\|_{\mathbb{H}}
\end{equation}for some $C_1,C_2>0$ depending on $\xi$ and $S_\mathbb{H}$, so that
there exists (a possibly smaller) $\beta>0$ such that $$0<C(\beta)< \det(J_h)\leq \frac{1}{C(\beta)}$$ whenever $\|h\|_{\mathbb{H}}<\beta$. Moreover, the properties of the cut-off $\xi$ and the boundary conditions $h(\pm1)=0$ yield
\begin{equation}\label{diff-prop-exp} \begin{aligned}
 &\varphi_h(x,y,z)= (x,y, z+h(y)) \quad \text{for} \quad (x,y,z)\in\mathcal{\partial B},\\[5pt]
 &\varphi_h(x, \pm 1, z)= (x,\pm 1, z +\xi(x,z)h(\pm 1))=(x, \pm 1, z) \ \text{for} \ (x,z)\in[-R,R]\times [-1,1],
 \\[5pt]
\end{aligned}\end{equation}Moreover, $\phi_h$ is the identity map on each of the four faces $\overline{\Gamma}_x^\pm$ and $\overline{\Gamma}_z^\pm$
Hence, the $C^1$-regularity,  \eqref{detJeta} and \eqref{diff-prop-exp} prove that $\varphi_h$ is a $C^1$-diffeomorphism from $\overline{\Omega}_0$ to $\overline{\Omega}_h$ that satisfies \eqref{diff-prop}.
\end{proof}

\begin{remark}
The diffeomorpshim \eqref{diffeo} is the most direct that can be introduced to transform $\Omega_h$ into $\Omega_0$ while ensuring \eqref{diff-prop}. Its regularity is the same as that of $h$, but this does not represent the optimal choice. To improve regularity, it is necessary to introduce a so-called \emph{regularizing} diffeomorphism that permits to gain half derivative compared to the regularity of $h$, see \cite{BocCasGan} for an example. However, we will show later that $h\in\mathbb{H}^4$, which provides us with enough regularity to carry out the analysis working with \eqref{diffeo}.
\end{remark}

Assuming that $\|h\|_{\mathbb{H}}< \beta$ with $\beta>0$ as in Lemma \ref{lemma-diffeo}, by means of the diffeomorphism $\varphi_h$ in \eqref{diffeo}, we transform \eqref{eq_main} in $\Omega_h$ as an elliptic problem in $\Omega_0$. After
introducing the matrices
\begin{equation}\label{matrices-diffeo}
\begin{aligned}
    A_h= \det(J_h) M_h^TM_h \quad \text{and} \quad B_h= \det(J_h) M_h \quad \text{with} \quad  M_h = (J_h^{-1})^T,\\[5pt]
\end{aligned}
\end{equation}we obtain that $(U,P)=(u, p)\circ \varphi_h$ solves the elliptic problem
\begin{equation}\label{BVprob_trans}
		\begin{aligned}
			&-\eta\nabla \cdot\left( A_h \nabla U \right)  +  U\cdot B_h\nabla U+ B_h\nabla P=0, \quad
			\nabla\cdot (B_h^TU)=0\quad\text{in}\quad \Omega_0\\[2pt]
			&U=0 \quad \text{on} \quad  \partial \mathcal{B} ,\quad U=\gamma Ve_1\quad \text{on} \quad \overline{\Gamma_{y}^\pm \setminus \mathcal{B}_{\pm 1}}\cup\Gamma_z^\pm\\[2pt]
			&U= \gamma V_ie_1\quad \text{on} \quad \Gamma_x^-,\quad U= \gamma V_oe_1\quad\text{on} \quad \Gamma_x^+.\\[5pt]
		\end{aligned}
	\end{equation}

We are now ready to study the dependence of $L(y, h)$, defined in \eqref{load},  on $h$  while keeping $y$ fixed. In the next proposition, we prove its Lipschitz continuity.

\begin{proposition}\label{L-incre}
	There exists $\beta_0\in (0, \beta]$ with $\beta$ as in Lemma \ref{lemma-diffeo} such that $ h\mapsto L( y, h)$ is Lipschitz continuous in $\mathbb{B}_{\beta_0}$. More precisely, for any $h_1,h_2 \in \mathbb{B}_{\beta_0}$,
\begin{equation}\label{lip-lift}
    \|L(\cdot, h_1) - L(\cdot, h_2)\|_{L^\infty(-1,1)} \leq  \gamma C(\beta_0) \|h_1-h_2\|_{\mathbb{H}^4}
\end{equation}with $\gamma \in [0,\gamma_0)$ for some $\gamma_0\in(0,\overline{\gamma}]$ and $\overline{\gamma}$ as in Theorem \ref{the_ex_regu}.
\end{proposition}
\begin{proof}
With $\varphi_h$ at hand,   the lift in \eqref{load}
can be written in terms of the $h$-independent fluid domain $\Omega_0$, namely,
    \begin{equation} \label{lift-trans}\begin{aligned}
        L(y,h)&=-e_3\cdot \int_{\partial \mathcal{B}_y}\widetilde{\mathbb{T}}_h(U, P) M_hn_0\\&=-e_3 \cdot \int_{\partial \mathcal{B}_y}  \left[\eta (M_h \nabla U + (M_h\nabla U)^T) - P\mathbb{I}\right] M_hn_0
        \end{aligned}
    \end{equation}
where $(U,P)$ solves \eqref{BVprob_trans} and $n_0$ is the  unit normal vector to $\partial \mathcal{B}$ pointing inside $\mathcal{B}$. In \eqref{lift-trans}, we have used the fact $\det(J_h)=1$ on $\partial \mathcal{B}_y$. We now show that, there exists $\beta_0>0$ such that
\begin{equation}\label{claim}\vspace{0.1em}
h\mapsto (U(h),P(h))\in W^{2, 2+\sigma}(\Omega_0) \times W^{1, 2+\sigma}(\Omega_0)\mbox{ belongs to }  C^1(\mathbb{B}_{\beta_0}).\vspace{0.1em}
\end{equation}
 To this end, with $\beta$ as in Lemma \ref{lemma-diffeo} we express \eqref{BVprob_trans} as
\begin{equation}\label{implicit}
\mathcal{H}(h,U, P)=0,
\end{equation}where  $\mathcal{H} :  \mathbb{B}_\beta\times  X \rightarrow Y$  with
\begin{equation*}\begin{aligned}
&X=W^{2, 2+\sigma}(\Omega_0)\times W^{1, 2+\sigma}(\Omega_0),\\
&Y=\bigg\{(\chi_1,\chi_2,\chi_3)\in L^{2+\sigma}(\Omega_0)\times W^{1, 2+\sigma}(\Omega_0)\times W^{2-\tfrac{1}{2+\sigma}, 2+\sigma}(\partial \Omega_0)  \\
&\qquad \qquad  \text{such that} \quad \int_{\Omega_0} \chi_2 = \int_{\partial \Omega_0}B_h^T \chi_3 \cdot n \bigg\}, \end{aligned}\end{equation*} and the components of $\mathcal{H}$ read
\begin{equation}\label{defH}
	\begin{aligned}
	\mathcal{H}_1(h, U, P)= &-\eta\nabla \cdot (A_h\nabla U) + U \cdot B_h\nabla U + B_h \nabla P,\\[5pt]
	\mathcal{H}_2(h, U, P)= & \ \nabla \cdot (B_h^T U), \\[5pt]
    \mathcal{H}_3(h, U, P)= & \ U_{|_{\partial \Omega_0}} - u_*,
	\end{aligned}
\end{equation}where $u_*$ is as in \eqref{u*}. We explicitly compute the matrices in \eqref{matrices-diffeo}, which read
\begin{equation}\label{exp-matrices}\begin{aligned}
 &A_h= \mathbb{I}  +\left(\begin{matrix}
\partial_z \xi h
  & 0 & -\partial_x \xi h \\[5pt]0&\partial_z \xi h & - \xi h'\\[5pt]
-\partial_x \xi h& - \xi h'&\frac{-\partial_z\xi h + (\partial_x \xi h)^2 + (\xi h')^2}{1+ \partial_z \xi h}
\end{matrix}\right), \
 &B_h= \mathbb{I} +\left(\begin{matrix}
\partial_z \xi h  & 0&-\partial_x \xi h \\[5pt]0 & \partial_z \xi h&  - \xi h'\\[5pt]
0 & 0& 0
\end{matrix}\right).\\[5pt]
\end{aligned}
\end{equation}
First, we have that $h\mapsto\mathcal{H}(h, U, P)$ is $C^1(\mathbb{B}_{\beta})$ since Lemma \ref{lemma-diffeo} yields the same regularity for the maps $h\mapsto A_h$, $h\mapsto B_h$ and $h\mapsto B_h^T$. Second, $(U,P)\mapsto \mathcal{H}(h, U,P)$ is $C^1(X)$ for any $h\in \mathbb{B}_{\beta}$; indeed, the linear differential operator $\mathcal{L}=D_{(U, P)}\mathcal{H}[h,U,P]$, whose components read
\begin{equation*}\begin{aligned}
\mathcal{L}_1( \chi, \Pi)= &-\eta\nabla \cdot (A_h\nabla \chi) + \chi \cdot B_h\nabla U + U \cdot B_h\nabla \chi+ B_h\nabla \Pi,\\[5pt]
\mathcal{L}_2( \chi, \Pi)= & \ \nabla \cdot (B_h^T\chi),\\[5pt]
\mathcal{L}_3( \chi, \Pi)  = &  \ \chi_{|_{\partial \Omega_0}},
\end{aligned}
\end{equation*} is bounded from  $X$ to $Y$, hence it is Fréchet differentiable. Moreover, it is continuous with respect to $(h,U,P)\in \mathbb{B}_\beta\times X$.
Thus, $(h, U, P) \mapsto\mathcal{H}(h, U, P)$ belongs to $C^1(  \mathbb{B}_\beta\times X)$. \par Moreover, we claim that $\mathcal{L}$ is an isomorphism from $X$ to $Y$. To prove this assertion, let us consider $(\phi_1, \phi_2, \phi_3 )\in Y$ and the linear elliptic problem \begin{equation}\label{iso-eq}\mathcal{L}(\chi, \Pi)= (\phi_1, \phi_2, \phi_3), \end{equation} which can be written as the inhomogeneous Stokes-type system
\begin{equation}\label{stokes-mod}\begin{aligned}
&-\eta \Delta \chi + \nabla \Pi =  - \chi \cdot B_h\nabla U -U\cdot B_h\nabla \chi \\
&  \quad \qquad \qquad \qquad-\eta\nabla \cdot ((\mathbb{I}-A_h)\nabla \chi) + (\mathbb{I}-B_h)\nabla \Pi + \phi_1 \quad &&&&\mbox{in}  \quad \Omega_0 \\[5pt]\ &\nabla\cdot \chi  = \ \nabla \cdot (\mathbb{I}-B_h^T \chi) +\phi_2 \quad &&&&\mbox{in} \quad \Omega_0 \\[5pt]
&\chi_{|_{\partial \Omega_0}}= \ \phi_3.
\end{aligned}
\end{equation}
By \eqref{exp-matrices} and since $h\in \mathbb{B}_\beta$, there exist $C_1(\beta),C_2(\beta)>0$, with $C_{1}(\beta), C_2(\beta)\rightarrow 0$ as $\beta\rightarrow 0$, such that
$$\|\mathbb{I}-A_h\|_{W^{1, \infty}(\Omega_0)}\leq C_1(\beta) \quad \text{and} \quad \|\mathbb{I}-B_h\|_{W^{2,\infty}(\Omega_0)}\leq C_2(\beta).$$ We hence infer that
there exist $\beta_0\in (0, \beta]$ and $r_0>0$ such that, by a contraction argument,  \eqref{stokes-mod} admits a unique solution $(\chi, \Pi)\in X$ provided that
\begin{equation}\label{small-cond}
\|h\|_{\mathbb{H}^4}<\beta_0 \qquad \mbox{and} \quad 	\|U\|_{W^{2, 2+\sigma}(\Omega_0)}< r_0.
\end{equation} Moreover, the properties of the diffeomorpshim $\varphi_h$ imply that
\begin{equation}\label{controlU-u}\begin{aligned}
 c\|u\|_{W^{2, 2+\sigma}(\Omega_h)}\leq &\|U\|_{W^{2, 2+\sigma}(\Omega_0)}\leq C\|u\|_{W^{2, 2+\sigma}(\Omega_h)},\\[5pt] c\|p\|_{W^{1, 2+\sigma}(\Omega_h)}\leq &\|P\|_{W^{1, 2+\sigma}(\Omega_0)}\leq C\|p\|_{W^{1, 2+\sigma}(\Omega_h)},
 \end{aligned}
\end{equation} with constants $0<c\leq C$ independent of $h$, by taking $\gamma \in (0,\overline{\gamma})$ sufficiently small, the bound \eqref{reg-est}
yields the second condition in \eqref{small-cond}. Thus, $\mathcal{L}$ is an isomorphism and there exists $C>0$, independent of $h$, such that the unique solution to \eqref{iso-eq} satisfies
\begin{equation}\label{est-iso}\begin{aligned}
    &\|\chi\|_{W^{2, 2+\sigma}(\Omega_0)} + \|\Pi\|_{W^{1, 2+\sigma}(\Omega_0)} \\[5pt]&\leq C \left( \|\phi_1\|_{L^{2+\sigma}(\Omega_0)} + \|\phi_2\|_{W^{1, 2+\sigma}(\Omega_0)} + \|\phi_3\|_{W^{2-\frac{1}{2+\sigma}, 2+\sigma}(\partial \Omega_0)}\right)\end{aligned}
\end{equation}
Therefore,  applying the Implicit Function Theorem to \eqref{implicit}, we conclude \eqref{claim}.

In addition, it follows from \eqref{claim} that there differential operators $D_h U[h] : \mathbb{B}_{\beta_0}\rightarrow W^{2, 2+\sigma}(\Omega_0)$ and $D_h P[h] : \mathbb{B}_{\beta_0}\rightarrow W^{1,2+\sigma}(\Omega_0)$ are bounded. We also know from \eqref{implicit} that, for any $\tilde h\in \mathbb{B}_{\beta_0}$,  $$U_h(\tilde{h}):=D_h U[h] \tilde{h}\quad \text{and} \quad P_h(\tilde{h}):=D_h P[h] \tilde{h}$$ solve the linear elliptic problem
\begin{equation*}
\mathcal{L}( U_h(\tilde{h}), P_h(\tilde{h}) )= - D_h\mathcal{H}[h, U, P]\tilde{h},
\end{equation*} which explicitly reads
\begin{equation*}
\begin{aligned}
&-\eta\nabla \cdot (A_h\nabla U_h(\tilde{h})) + U_h(\tilde{h}) \cdot B_h\nabla U + U \cdot B_h\nabla U_h(\tilde{h}) + B_h\nabla P_h(\tilde{h})= F(h, \tilde{h}, U, P)
&\mbox{in} \ \Omega_0 \\[5pt]
&\nabla \cdot (B_h^T U_h(\tilde{h})) = G(h, \tilde{h}, U) &\mbox{in} \ \Omega_0\\[5pt]
&U_h(\tilde{h})_{|_{\partial \Omega_0}}=0
\end{aligned}
\end{equation*}
with the source terms
\begin{equation*}\begin{aligned}
    &F(h, \tilde{h}, U, P)=-\eta \nabla \cdot (A_h(\tilde{h})\nabla U) + U\cdot B_h(\tilde{h})\nabla U + B_h(\tilde{h})\nabla P,\\[5pt]&
    G(h, \tilde{h}, U)= \nabla \cdot ( B^T_h(\tilde{h}) U)
    \end{aligned}
\end{equation*}
and
\begin{equation*}
A_h(\tilde{h})=- D_h A_h[h]\tilde{h}, \quad
B_h(\tilde{h})= -D_h B_h[h]\tilde{h},\quad
B_h^T(\tilde{h})= -D_h B_h^T[h]\tilde{h}, \\[5pt]
\end{equation*} where $D_h A_h[h]$, $ D_h B_h[h]$ and $D_h B_h^T[h]$ are the differential operators associated with the maps $h\mapsto A_h$, $h\mapsto B_h$ and $h\mapsto B_h^T$.
Since $\mathcal{L}$ is an isomorphism, we know that $(U_h(\tilde{h}), P_h(\tilde{h}))$ is uniquely determined and, by \eqref{est-iso}, it satisfies
\begin{equation}\label{est-appdiff}\begin{aligned}
    \|U_h(\tilde{h})\|_{W^{2, 2+\sigma}(\Omega_0)} &+ \|P_h(\tilde{h})\|_{W^{1, 2+\sigma}(\Omega_0)} \\[5pt]&\leq C \left( \|F(h, \tilde{h}, U, P)\|_{L^{2+\sigma}(\Omega_0)} + \|G(h, \tilde{h}, U)\|_{W^{1, 2+\sigma}(\Omega_0)} \right).\end{aligned}
\end{equation}
For any $h\in \mathbb{B}_{\beta_0}$, it follows from the explicit expressions \eqref{exp-matrices} that there exists $C(\beta_0)>0$ such that
\begin{equation*}\begin{aligned}
    \|A_h(\tilde{h})\|_{H^2(\Omega_0)}\leq C(\beta_0)&\|\tilde{h}\|_{\mathbb{H}^4}, \qquad   \|B_h(\tilde{h})\|_{H^2(\Omega_0)}\leq C(\beta_0)\|\tilde{h}\|_{\mathbb{H}^4},
    \end{aligned}
\end{equation*} so that, using the continuous embedding $H^2(\Omega_0)\subset L^\infty(\Omega_0)$, the right-hand side of \eqref{est-appdiff} can be estimated by
\begin{equation}\label{est-sources-dif}
    \begin{aligned}
        &\|F(h, \tilde{h}, U, P)\|_{L^{2+\sigma} (\Omega_0)} + \|G(h, \tilde{h}, U)\|_{W^{1,2+\sigma} (\Omega_0)}\\[5pt]&\leq C\left(\|A_h(\tilde{h})\|_{H^2(\Omega_0)}+\|B_h(\tilde{h})\|_{H^2(\Omega_0)}\right)\left( \| U\|_{W^{2, 2+\sigma}(\Omega_0)} + \| P\|_{W^{1, 2+\sigma}(\Omega_0) } \right)
     \\[5pt]&  \leq  C(\beta_0)\|\tilde{h}\|_{\mathbb{H}^4}\left(\| U\|_{W^{2, 2+\sigma}(\Omega_0)} + \| P\|_{W^{1, 2+\sigma}(\Omega_0)} \right).
    \end{aligned}
\end{equation}
Due to \eqref{controlU-u} and \eqref{reg-est}, we then obtain
\begin{equation*}
    \begin{aligned}
    \|U_h(\tilde{h})\|_{W^{2, 2+\sigma}(\Omega_0)} + \|P_h(\tilde{h})\|_{W^{1, 2+\sigma}(\Omega_0)} \leq C(\beta_0)\gamma  \|\tilde{h}\|_{\mathbb{H}^4},
    \end{aligned}
\end{equation*}which, in turn, implies the uniform bound for the norms of the differential operators $D_hU[h]$ and $D_hP[h]$
\begin{equation}\label{uni-bou-diff}
    \sup_{h\in \mathbb{B}_{\beta_0}} \|D_hU[h]\| +  \sup_{h\in \mathbb{B}_{\beta_0}} \|D_hP[h]\| \leq C(\beta_0)\gamma.
\end{equation}

Finally, let us estimate the variation of the lift, defined in \eqref{lift-trans}, with respect to the variation of $h$. Using the same argument and the Trace Theorem as in the proof of Lemma \ref{lem-Lbound}, we have that,  for any $h_1,h_2\in \mathbb{B}_{\beta_0}$,
\begin{equation}\label{diff-lift}\begin{aligned}
&|	L(y, h_1) - 	L(y, h_2) |\\[5pt]&=
\bigg|\int_{\partial \mathcal{B}_y}  \left(\widetilde{\mathbb{T}}_{h_1}(U(h_1), P(h_1)) M_{h_1}n_0 - \widetilde{\mathbb{T}}_{h_2}(U(h_2), P(h_2)) M_{h_2}n_0\right) \bigg|
\\[5pt]&
=\bigg| \int_{\partial \mathcal{B}_y}\bigg[\left(\widetilde{\mathbb{T}}_{h_1}(U(h_1),P(h_1)) - \widetilde{\mathbb{T}}_{h_2}(U(h_2),P(h_2)) \right) M_{h_1}n_0 \\[2pt]&\qquad \qquad\qquad  +\widetilde{\mathbb{T}}_{h_2}(U(h_2),P(h_2)) (M_{h_1} - M_{h_2})n_0\bigg]\bigg|\\[5pt]&
\leq C(\beta_0)\bigg(\|U(h_1)- U(h_2)\|_{W^{2, 2+\sigma}(\Omega_0)}+ \|P(h_1)-P(h_2)\|_{W^{1, 2+\sigma}(\Omega_0)} \\[2pt]&\qquad\qquad  +\left(\|U(h_2)\|_{W^{2,2+\sigma}(\Omega_0)}+ \|P(h_2)\|_{W^{1, 2+\sigma}(\Omega_0)}\right)\|h_1-h_2\|_{\mathbb{H}^4}\bigg).\\[5pt]
\end{aligned}
\end{equation}
 Combining the Mean Value Inequality with \eqref{uni-bou-diff} and using \eqref{controlU-u} together with \eqref{reg-est}, we obtain that
\begin{equation*}\begin{aligned} &\|	L(\cdot, h_1) - 	L(\cdot, h_2) \|_{L^\infty(-1,1)}\\[5pt]&\leq C(\beta_0)\left(\gamma + \|u(h_2)\|_{W^{2, 2+\sigma}(\Omega_{h_2})}+ \|p(h_2)\|_{W^{1, 2+\sigma}(\Omega_{h_2})}\right)\|h_1-h_2\|_{\mathbb{H}^4}\\[5pt]&\leq \gamma C(\beta_0)\|h_1-h_2\|_{\mathbb{H}^4},
\end{aligned}
\end{equation*} which proves the Lipschitz continuity of $L(\cdot, h)$ with respect to $h$.
\end{proof}

\subsection{Proof of Theorem \ref{main-theo}}\label{sec-proof}

We are now in position to set the fixed point method to uniquely solve the coupled system \eqref{BVprobR}-\eqref{beam-pb} and, hence, to prove Theorem \ref{main-theo}.
Consider a given $\overline{h}\in  \mathbb{H}^4\subset   C^{1,1}_0([-1,1])$. Theorem \ref{the_ex_regu} yields a unique pair $(u(\overline{h}), p(\overline{h}))\in W^{2, 2+\sigma}(\Omega_0)\times W^{1, 2+\sigma}(\Omega_0)$ that solves \eqref{BVprobR} with inflow/outflow magnitude $\gamma\in [0, \overline{\gamma})$. The lift $L(\cdot, \overline{h})$ generated by the solution $(u(\overline{h}), p(\overline{h}))$  belongs to $C(-1,1)$ due to Lemma \ref{lem-Lbound} and we view $L(\cdot, \overline{h})$ as the source term of the beam equation \eqref{mod-beam-pb}. Thanks to Proposition \ref{prop-exi-beam}, it admits a unique solution that depends on $\overline{h}$. This defines the map
\begin{equation}\label{fixedpoint-map}
   \mathfrak{h} : \qquad  \begin{array}{ccc}
       \mathbb{B}_{\beta_0}\subset \mathbb{H}^4 & \rightarrow & \mathbb{H}^4\\[5pt]
        \overline{h}& \mapsto & \mathfrak{h}(\overline{h})
    \end{array}
\end{equation}
where $\mathfrak{h}(\overline{h})$ is the unique solution to \eqref{mod-beam-pb}. To continue, we show that $\mathfrak{h}$ is a contraction map on $\mathbb{B}_{\beta_0}$.\par
First, we show that $\mathfrak{h}(\mathbb{B}_{\beta_0})\subset \mathbb{B}_{\beta_0}$ for $\gamma$ sufficiently small. In fact, given any $\overline{h}\in \mathbb{B}_{\beta_0}\subset C^{1,1}_0([-1,1])$, we know from Proposition \ref{prop-exi-beam} that $\mathfrak{h}(\overline{h})\in \mathbb{H}^4$ and, for $\gamma\in [0, \overline{\gamma})$,

\begin{equation*}
    \|\mathfrak{h}(\overline{h})\|_{\mathbb{H}^4}
    \leq C\gamma;
\end{equation*}this implies that there exists $\gamma_0\in (0, \overline{\gamma}]$ such that
\begin{equation}\label{bound-map}
    \| \mathfrak{h}(\overline{h}) \|_{\mathbb{H}^4} < \beta_0,
\end{equation}whenever $\gamma\in [0, \gamma_0)$.
 Consider $\mathfrak{h}(\overline{h}_1)$ and $\mathfrak{h}(\overline{h}_2)$ solutions to \eqref{mod-beam-pb} with source term $L(\cdot, \overline{h}_1)$ and $L(\cdot, \overline{h}_2)$, respectively. Their difference $\mathfrak{h}(\overline{h}_1)-\mathfrak{h}(\overline{h}_2)$ satisfies
\begin{equation}\label{beam-diff}\begin{aligned}
    &\left(\mathfrak{h}(\overline{h}_1) - \mathfrak{h}(\overline{h}_2)\right)'''' + f(\mathfrak{h}(\overline{h}_1)) - f(\mathfrak{h}(\overline{h}_2)) = L(\cdot, \overline{h}_1) - L(\cdot, \overline{h}_2) \qquad \text{in} \quad (-1,1)\\[2pt]
    &\left(\mathfrak{h}(\overline{h}_1) - \mathfrak{h}(\overline{h}_2)\right)(\pm 1) = 0 \\[2pt]&\left(\mathfrak{h}(\overline{h}_1) - \mathfrak{h}(\overline{h}_2)\right)'(\pm 1)=0 \ \text{ or } \  \left(\mathfrak{h}(\overline{h}_1) - \mathfrak{h}(\overline{h}_2)\right)''(\pm 1)=0.\\[5pt]
		\end{aligned}
\end{equation}
 Testing \eqref{beam-diff} with $\mathfrak{h}(\overline{h}_1)-\mathfrak{h}(\overline{h}_2)$ yields
 \begin{equation*}\begin{aligned}
     \|\mathfrak{h}(\overline{h}_1)-\mathfrak{h}(\overline{h}_2) \|^2_{\mathbb{H}} \ + &\int_{-1}^1 \left(f(\mathfrak{h}(\overline{h}_1))-f(\mathfrak{h}(\overline{h}_2))\right) (\mathfrak{h}(\overline{h}_1)-\mathfrak{h}(\overline{h}_2))\\[5pt]&= \int_{-1}^{1}\left(L(\cdot, \overline{h}_1)-L(\cdot, \overline{h}_2)\right)(\mathfrak{h}(\overline{h}_1)-\mathfrak{h}(\overline{h}_2))
     \end{aligned}
 \end{equation*}
which, due to \eqref{assf}, \eqref{SH-embedding} and  \eqref{lip-lift}, implies that
\begin{equation}\label{conv-Hbb}
 \|\mathfrak{h}(\overline{h}_1)-\mathfrak{h}(\overline{h}_2) \|_{\mathbb{H}} \leq 2S_{\mathbb{H}}\|L(\cdot, \overline{h}_1)-L(\cdot, \overline{h}_2)\|_{L^\infty(-1,1)}\leq C \gamma \|\overline{h}_1 - \overline{h}_2\|_{\mathbb{H}^4}
\end{equation}
By elliptic regularity for \eqref{beam-diff}, by \eqref{assf} and the continuous embedding $\mathbb{H}\subset L^\infty(-1,1)$, it follows that
\begin{equation}\label{conv-liftH4}\begin{aligned}
        &\|\mathfrak{h}(\overline{h}_1)- \mathfrak{h}(\overline{h}_2)\|_{\mathbb{H}^4} \\[5pt]&\leq C\left( \|f(\mathfrak{h}(\overline{h}_1)) - f(\mathfrak{h}(\overline{h}_2))\|_{L^2(-1,1)} + \|L(\cdot, \overline{h}_1)- L(\cdot, \overline{h}_2)\|_{L^2(-1,1)}\right)\\[5pt]
        &\leq C\left(  \|\mathfrak{h}(\overline{h}_1)-\mathfrak{h}(\overline{h}_2) \|_{\mathbb{H}} + \|L(\cdot, \overline{h}_1)- L(\cdot, \overline{h}_2)\|_{L^\infty(-1,1)} \right)\\[5pt]&\leq C\gamma\|\overline{h}_1- \overline{h}_2\|_{\mathbb{H}^4},\\[5pt]
        \end{aligned}
\end{equation}where in the last inequality we used \eqref{conv-Hbb} and \eqref{lip-lift}. Then, we deduce from \eqref{conv-liftH4} that there exists  $\gamma_0\in (0, \overline{\gamma}]$, possibly smaller than the one implying \eqref{bound-map}, such that
\begin{equation}\label{contraction-est}
    \|\mathfrak{h}(\overline{h}_1)- \mathfrak{h}(\overline{h}_2)\|_{\mathbb{H}^4}\leq M \|\overline{h}_1- \overline{h}_2\|_{\mathbb{H}^4}
\end{equation} with $0<M<1$. Finally, gathering \eqref{bound-map} and \eqref{contraction-est} together, we conclude that
the map \eqref{fixedpoint-map} is a contraction on $\mathbb{B}_{\beta_0}$ for $\gamma\in [0, \gamma_0)$.  Banach fixed-point theorem then implies the existence and uniqueness of a  fixed point, that is, there exists a unique $h\in \mathbb{B}_{\beta_0}$  such that $\mathfrak{h}(h)=h$. This is equivalent to say that there exists a unique solution $(h, u, p)\in \mathbb{B}_{\beta_0}\times W^{2, 2+\sigma}(\Omega_h)\times W^{1, 2+\sigma}(\Omega_h)$ to the coupled system \eqref{BVprobR}-\eqref{beam-pb}.\par
We now show the Lipschitz continuity of $\gamma\mapsto (h(\gamma), u(\gamma), p(\gamma))$. To this end, let us consider for $\gamma_1, \gamma_2 \in [0, \gamma_0)$, the solutions $(h(\gamma_1), u(\gamma_1), p(\gamma_1))$ and $(h(\gamma_2), u(\gamma_1), p(\gamma_2))$ to the FSI problem  \eqref{BVprobR}-\eqref{beam-pb} with flow magnitude $\gamma_1$ and $\gamma_2$, respectively. In order to compare the two solutions to the fluid boundary-value problem \eqref{BVprobR}, which are defined in different domains $\Omega_{h(\gamma_1)}$ and $\Omega_{h(\gamma_2)}$, we transform one domain into the other through the diffeomorphism  in Lemma \ref{lemma-diffeo}.
Introducing the transformed velocity and  pressure $$U(\gamma_2)= u(\gamma_2)\circ \varphi_{h(\gamma_2) - h(\gamma_1)}, \qquad P(\gamma_2)= u(\gamma_2)\circ \varphi_{h(\gamma_2) - h(\gamma_1)},$$ defined in $\Omega_{h(\gamma_1)}$, it follows that $(U(\gamma_2), P(\gamma_2))$ solves  the transformed problem
\begin{equation}
		\begin{aligned}
			&-\eta\nabla \cdot\left( A_{h(\gamma_2)-h(\gamma_1)} \nabla U(\gamma_2) \right)  +  U(\gamma_2)\cdot B_{h(\gamma_2)-h(\gamma_1)}\nabla U(\gamma_2)\\[5pt]& \qquad \qquad \qquad \qquad \qquad \qquad + B_{h(\gamma_2)-h(\gamma_1) }\nabla P(\gamma_2)=0\quad\text{in}\quad \Omega_{h(\gamma_1)}, \\[5pt]
			&\nabla\cdot (B_{h(\gamma_2)-h(\gamma_1)}^TU(\gamma_2)))=0\quad\text{in}\quad \Omega_{h(\gamma_1)},\\[5pt]&	U(\gamma_2)=0 \quad \text{on} \quad  \partial \mathcal{B}^{h(\gamma_1)} ,\quad U(\gamma_2)=\gamma_2 Ve_1\quad \text{on} \quad \overline{\Gamma_{y}^\pm \setminus \mathcal{B}_{\pm 1}}\cup\Gamma_z^\pm,\\[5pt]
			&U(\gamma_2)= \gamma_2 V_ie_1\quad \text{on} \quad \Gamma_x^-,\quad U(\gamma_2)= \gamma_2 V_oe_1\quad\text{on} \quad \Gamma_x^+.
		\end{aligned}
\end{equation}

 The differences $W= u(\gamma_1)-U(\gamma_2)$ and $Q= p(\gamma_1)-P(\gamma_2)$ solve
\begin{equation}\label{WQ-pb}
		\begin{aligned}
		&-\eta\Delta  W +\nabla Q =
       - u(\gamma_1)\cdot \nabla W - W\cdot \nabla U(\gamma_2)\\[5pt]& \ \qquad \qquad \qquad \quad   + S_1\left(h(\gamma_2)- h(\gamma_1), U(\gamma_2), P(\gamma_2)\right)
        \quad\text{in}\quad \Omega_{h(\gamma_1)},  \\[5pt]
		&\nabla\cdot W =	S_2\left((h(\gamma_2)- h(\gamma_1), U(\gamma_2)\right)\quad\text{in}\quad \Omega_{h(\gamma_1)},\\[5pt]	&	W=0 \quad \text{on} \quad  \partial \mathcal{B}^{h(\gamma_1)},\quad W=(\gamma_1 - \gamma_2) Ve_1\quad \text{on} \quad \overline{\Gamma_{y}^\pm \setminus \mathcal{B}_{\pm 1}}\cup\Gamma_z^\pm,\\[5pt]
			&W= (\gamma_1 - \gamma_2) V_ie_1\quad \text{on} \quad \Gamma_x^-,\quad W = (\gamma_1 - \gamma_2) V_oe_1\quad\text{on} \quad \Gamma_x^+.
		\end{aligned}
\end{equation}
with
\begin{align*}
  &S_1\left(h(\gamma_2)- h(\gamma_1), U(\gamma_2), P(\gamma_2)\right)= - \eta \nabla \cdot \left((A_{h(\gamma_2)-h(\gamma_1)}- \mathbb{I})\nabla U(\gamma_2) \right)\\&\qquad \qquad+ U(\gamma_2)\cdot (B_{h(\gamma_2)- h(\gamma_1)} - \mathbb{I})\nabla U(\gamma_2)+ (B_{h(\gamma_2)- h(\gamma_1)}- \mathbb{I})\nabla P(\gamma_2),\\[5pt]
&S_2\left(h(\gamma_2)- h(\gamma_1), U(\gamma_2)\right)=
   \nabla\cdot \big((B_{h(\gamma_2)-h(\gamma_1))}^T- \mathbb{I})U(\gamma_2)\big).
\end{align*}
Applying \cite[Theorem 10.5.4]{Mazya-Rossmann2010} to \eqref{WQ-pb} then yields
\begin{equation}\label{regest-WQ}\begin{aligned}
   \|W\|_{W^{2, 2+\sigma}(\Omega_{h(\gamma_1)})}& + \|Q\|_{W^{1, 2+\sigma}(\Omega_{h(\gamma_1)})}\\[5pt]
\leq C(\beta_0)\bigg( &
  \|u(\gamma_1)\cdot \nabla W+ W\cdot \nabla U(\gamma_2)\|_{L^{2+\sigma}(\Omega_{h(\gamma_1)})}
+ |\gamma_1-\gamma_2|\\[5pt]&
+\|S_1\left(h(\gamma_2)- h(\gamma_1), U(\gamma_2), P(\gamma_2)\right)\|_{L^{2+\sigma}(\Omega_{h(\gamma_1)})} \\[5pt]&+\|S_2\left(h(\gamma_2)- h(\gamma_1), U(\gamma_2)\right)\|_{W^{1,2+\sigma}(\Omega_{h(\gamma_1)})}
\bigg).
\end{aligned}
\end{equation}
After combining \eqref{controlU-u} (up to obvious changes) with  \eqref{reg-est} and estimating the terms $S_1$, $S_2$ in the right-hand side of \eqref{regest-WQ} in the same way as in \eqref{est-sources-dif}, we infer that there exists $\gamma^*\in (0, \gamma_0]$, depending on $\beta_0$, such that for $\gamma_1, \gamma_2 \in [0, \gamma^*)$,

\begin{equation}\label{WQ-est}\begin{aligned}
    \|W\|_{W^{2, 2+\sigma}(\Omega_{h(\gamma_1)})} +& \|Q\|_{W^{1, 2+\sigma}(\Omega_{h(\gamma_1)})}\\[5pt]&\leq C(\beta_0)\left( |\gamma_1-\gamma_2| + \gamma_2\|h(\gamma_1)- h(\gamma_2)\|_{\mathbb{H}^4}\right).
    \end{aligned}
\end{equation}
Fixing $y\in(-1,1)$, the variation of the lift is
\begin{equation*}\begin{aligned}
   & L(y, h(\gamma_1))- L(y, h(\gamma_2)) \\[5pt]&= \int_{\partial \mathcal{B}_y^{h(\gamma_1)}}\bigg( \mathbb{T}(u(\gamma_1), p(\gamma_1)) n_{h(\gamma_1)}  - \widetilde{\mathbb{T}}_{h(\gamma_2)-h(\gamma_1)}(U(\gamma_2), P(\gamma_2)) M_{h(\gamma_2)-h(\gamma_1)} n_{h(\gamma_1)}\bigg),
    \end{aligned}
\end{equation*} where $n_{h(\gamma_1)}$ is the unit normal vector to $\partial \mathcal{B}^{h(\gamma_1)}$ pointing inside $\mathcal{B}^{h(\gamma_1)}.$
Arguing as in \eqref{diff-lift} and by \eqref{WQ-est},
we get the estimate
\begin{equation}\label{diff-lift-gamma}\begin{aligned}
  & \| L(\cdot, h(\gamma_1))- L(\cdot, h(\gamma_2))\|_{L^\infty(-1,1)}\\[5pt]&\leq C(\beta_0)\left( \|W\|_{W^{2, 2+\sigma}(\Omega_{h(\gamma_1)})} + \|Q\|_{W^{1, 2+\sigma}(\Omega_{h(\gamma_1)})} + \gamma_2 \|h(\gamma_1)- h(\gamma_2)\|_{\mathbb{H}^4}    \right) \\[5pt]&\leq C(\beta_0)\left( |\gamma_1-\gamma_2|+ \gamma_2 \|h(\gamma_1)- h(\gamma_2)\|_{\mathbb{H}^4}    \right) .\\[5pt]
  \end{aligned}
\end{equation}
Since $h(\gamma_1)$ and $h(\gamma_2)$ solve the beam problem \eqref{beam-pb}, it follows that
\begin{equation}\label{beam-pb-diff}
		\begin{aligned}
			&(h(\gamma_1)- h(\gamma_2))''''+ f(h(\gamma_1)) -f(h(\gamma_2)) = L(\cdot, h(\gamma_1)) - L(\cdot, h(\gamma_2))   \ \ \text{in} \ \ (-1,1) \\[5pt]
			&(h(\gamma_1)- h(\gamma_2))(\pm 1) = 0 \\[5pt]  &(h(\gamma_1)- h(\gamma_2))'(\pm 1)=0 \ \text{ or } \ (h(\gamma_1)- h(\gamma_2))''(\pm 1)=0.\\[5pt]
		\end{aligned}
	\end{equation}
Analogously to \eqref{conv-liftH4}, by elliptic regularity for \eqref{beam-pb-diff}, exploiting \eqref{assf} and \eqref{diff-lift-gamma}, we derive
\begin{equation*}\begin{aligned}
    \|h(\gamma_1)-h(\gamma_2)\|_{\mathbb{H}^4} &\leq C\| L(\cdot, h(\gamma_1))- L(\cdot, h(\gamma_2))\|_{L^\infty(-1,1)}\\[5pt]&\leq C(\beta_0)\left( |\gamma_1-\gamma_2|+ \gamma_2 \|h(\gamma_1)- h(\gamma_2)\|_{\mathbb{H}^4}    \right),
    \end{aligned}
\end{equation*}which implies that there exists (a possibly smaller) $\gamma^*\in (0, \gamma_0)$ such that
\begin{equation*}
\|h(\gamma_1)-h(\gamma_2)\|_{\mathbb{H}^4}\leq C(\beta_0) |\gamma_1 - \gamma_2| \quad \text{for any} \quad \gamma_1, \gamma_2\in [0, \gamma^*).
\end{equation*} Consequently, by \eqref{WQ-est}, we also infer that
\begin{equation*}
\|u(\gamma_1)-U(\gamma_2)\|_{W^{2, 2+\sigma}(\Omega_{h(\gamma_1)})} +\|p(\gamma_1)-P(\gamma_2)\|_{W^{1, 2+\sigma}(\Omega_{h(\gamma_1)})} \leq C(\beta_0) |\gamma_1 - \gamma_2|.
\end{equation*}

The proof is so complete.

\subsection{The symmetric case}\label{symmetry}

We prove here that, under suitable symmetry assumptions, Theorems \ref{the_ex_regu} and \ref{main-theo} slightly modify.\par
Not only we assume the necessary condition \eqref{flux}, but we also assume that, for some
$$V\in W^{2-\tfrac{1}{2+\sigma}, 2+\sigma }(-1,1)^2\quad\mbox{s.t.}\quad
V(y,\pm1)=0\mbox{ for  }|y|\le1\mbox{ and }V(\pm1 ,z)=0\mbox{ for  }|z|\le1,$$
the inflow/outflow satisfy $V_i, V_o \in W^{2-\tfrac{1}{2+\sigma}, 2+\sigma }(-1,1)^2$ and
\begin{equation}\label{symminout}
V_i(y,z)=V_o(y,z)=V(y,z)\mbox{ on }\partial[-1,1]^2,\qquad V_i(y,z)=V_i(y,-z)\mbox{ for }|z|\le1.
\end{equation}
We also need the beam to be symmetric with respect to the plane $z=0$, that is,
\begin{equation}\label{symbeam}
(x,y,z)\in\mathcal{B}\ \Longleftrightarrow\ (x,y,-z)\in\mathcal{B}.
\end{equation}

Under these assumptions, Theorem \ref{the_ex_regu} becomes the following statement that we quote without proof since it can be obtained by
arguing as in \cite{clara,GS}.

\begin{theorem}\label{symmNS} Assume \eqref{symminout}-\eqref{symbeam} and let $h\equiv0$. For any $\gamma\geq0$, \eqref{eq_main} admits a weak solution $(u, p)$ in $\Omega_h$ and (at least)
one of them is symmetric, that is, for a.e.\ $(x,y,z)\in\Omega$
\begin{equation}\label{symmsol}
\begin{array}{ll}
u_1(x,y,z)=u_1(x,y,-z),\qquad & u_2(x,y,z)=u_2(x,y,-z),\\
u_3(x,y,z)=-u_3(x,y,-z),\qquad & p(x,y,z)=p(x,y,-z).
\end{array}
\end{equation}
There exists $\overline{\gamma}>0$ such that the weak solution is unique for $\gamma\in [0, \overline{\gamma})$ and it satisfies \eqref{symmsol}. Moreover,
any weak solution is a strong solution, according to Definition \ref{def_1}.
\end{theorem}

Clearly, also \eqref{reg-est} holds. By combining Theorems \ref{main-theo} and \ref{symmNS}, we obtain the following statement.

\begin{theorem}\label{main-theo-symm}
Assume \eqref{symminout}-\eqref{symbeam}. There exists $\beta_0>0$ and $\gamma_0\in (0, \overline{\gamma})$ such that for any $\gamma \in [0,\gamma_0)$, the FSI problem \eqref{BVprobR}-\eqref{beam-pb} admits a unique solution $$(h, u, p)\in\mathbb{B}_{\beta_0}\times W^{2, 2+\sigma}(\Omega_h)\times W^{1, 2+\sigma}(\Omega_h)$$
and such solution is such that
$$h\equiv0,\qquad (u,p)\mbox{ satisfy \eqref{symmsol}}.$$
In particular, with the notation in \eqref{load}, $L(y, u(h), p(h))\equiv0$ in $[-1,1]$.
\end{theorem}
\begin{proof} It follows the lines of the proof of \cite[Theorem 2.1]{clara} with one crucial difference, in \eqref{BVprobR}-\eqref{beam-pb} the beam is {\em free to move}.
So, let us just explain how to overcome this difficulty.\par
First, we {\em assume} that $h\equiv0$. Then Theorem \ref{symmNS} tells us that $(u,p)$ satisfies \eqref{symmsol}. For any such pair, we have $L(y, u(h), p(h))\equiv0$ in $[-1,1]$ and, hence,
instead of \eqref{compatibility}, the corresponding $h$ satisfies
$$
h''''(y)+f\big(h(y)\big)=0\quad\mbox{in }(-1,1)
$$
plus the boundary conditions. From \eqref{assf} this implies that $h\equiv0$, thereby confirming the initial guess. Since Theorem \ref{main-theo} ensures uniqueness of
$(h, u, p)$ solving \eqref{BVprobR}-\eqref{beam-pb}, the triple $(0,u(0),p(0))$ is the unique solution.\end{proof}

The last statement in Theorem \ref{main-theo-symm} has a clear (and surprising) physical interpretation: {\em there is no lift force acting on a symmetric (clamped or hinged) beam in a wind tunnel,
as long as the (symmetric) inflow/outflow remains sufficiently small}.

\section*{Appendix}\label{sec_comp}

	Let us first derive the following upper bounds for the constant $S_\mathbb{H}$ in \eqref{SH-embedding}:
    \begin{center}
(i) \ if $\mathbb{H}= H_0^2(-1,1)$, then $S_\mathbb{H}\le\frac{\sqrt{2}}{2}$; \quad
(ii) \ if $\mathbb{H}= H^2\cap H_0^1(-1,1)$, then
$S_{\mathbb{H}} \leq \frac{4\sqrt{2}}{3}.$
\end{center}

\begin{proof}
Suppose that $\mathbb{H}= H_0^2(-1,1)$. From the Fundamental Theorem, for any $x\in(-1,1)$ we have
\begin{equation}\label{eq_8}
h(x)=\int_{-1}^xh'(t)\,{\rm d}t=\int_{-1}^x\int_{-1}^th''(s)\,{\rm d}s\,{\rm d}t
\end{equation}
and similarly
\begin{equation}\label{eq_9}
h(x)=-\int_{-1}^x\int_t^1h''(s)\,{\rm d}s\,{\rm d}t.
\end{equation}
Summing \eqref{eq_8} with \eqref{eq_9}, we get
\begin{equation*}
2h(x)=\int_{-1}^x\int_{-1}^th''(s)\,{\rm d}s\,{\rm d}t-\int_{-1}^x\int_t^1h''(s)\,{\rm d}s\,{\rm d}t
\end{equation*}
hence
\begin{equation}\label{eq_10}\begin{aligned}
2|h(x)|&\le\int_{-1}^x\int_{-1}^t|h''(s)|\,{\rm d}s\,{\rm d}t+\int_{-1}^x\int_t^1|h''(s)|\,{\rm d}s\,{\rm d}t\\
&=\int_{-1}^x\int_{-1}^1|h''(s)|\,{\rm d}s\,{\rm d}t =(x+1)\left\|h''\right\|_{L^1(-1,1)}.
\end{aligned}
\end{equation}
Analogously, since
\begin{equation*}
h(x)=-\int_x^1\int_{-1}^th''(s)\,{\rm d}s\,{\rm d}t\hspace{.3in}\mbox{and}\hspace{.3in}h(x)=\int_x^1\int_t^1h''(s)\,{\rm d}s\,{\rm d}t
\end{equation*}
we find that
\begin{equation}\label{eq_11}
2|h(x)|\le(1-x)\left\|h''\right\|_{L^1(-1,1)}.
\end{equation}
Summing \eqref{eq_10} with \eqref{eq_11}, and using H\"older inequality we get
\begin{equation*}
|h(x)|\le\tfrac{1}{2}\left\|h''\right\|_{L^1(-1,1)}\le\tfrac{\sqrt{2}}{2}\left\|h''\right\|_{L^2(-1,1)}.
\end{equation*}
Since $x\in(-1,1)$ is arbitrary, we arrive at
\begin{equation}\label{eq_13}
\left\|h\right\|_{L^\infty(-1,1)}\le\tfrac{\sqrt{2}}{2}\left\|h''\right\|_{L^2(-1,1)}.
\end{equation}
and (i) is proved. \par
Suppose that $h\in H^2\cap H_0^1(-1,1)$. Since $h=0$ at $x=\pm1$ there exists $t_0\in[-1,1]$ be such that $h'(t_0)=0$. Then, from the Fundamental Theorem, for any $x\in(-1,1)$, and from H\"older inequality, we have
\begin{align*}
h(x)&=\int_{-1}^xh'(t)\,{\rm d}t=\int_{-1}^x\int_{t_0}^th''(s)\,{\rm d}s\,{\rm d}t\le\left\|h''\right\|_{L^2(-1,1)}\int_{-1}^1\sqrt{|t-t_0|}\,{\rm d}t.
\end{align*}
Since the function $t_0\mapsto\int_{-1}^1\sqrt{|t-t_0|}\,{\rm d}t$ attains its maximum at $t_0=\pm1$, we get
\begin{equation*}
\left\|h\right\|_{L^\infty(-1,1)}\le\tfrac{4\sqrt{2}}{3}\left\|h''\right\|_{L^2(-1,1)},
\end{equation*}
which proves (ii).\end{proof}

With other calculus tools we also prove that
\begin{center}
the norm $\|\cdot\|_{\mathbb{H}^4}$ defined in \eqref{H4norm}
is equivalent to \\the norm $\|\cdot\|_{H^4(-1,1)}$ in the closed subspace	$\mathbb{H}^4\subset H^4(-1,1)$.
\end{center}
\begin{proof} Let $h\in\mathbb{H}^4$. The definition \eqref{Hnorm} shows that $\|h''\|_{L^2(-1,1)}$ bounds both $\|h'\|_{L^2(-1,1)}$ and $\|h\|_{L^2(-1,1)}$. Hence, it suffices to prove that
$\|h''''\|_{L^2(-1,1)}$ bounds $\|h''\|_{L^2(-1,1)}$ and $\|h'''\|_{L^2(-1,1)}$.\par
The first bound follows from two integration by parts and the H\"older inequality:
\begin{eqnarray*}
\|h''\|_{L^2(-1,1)}^2 &=& -\int_{-1}^{1}h'''h'+[h''h']_{-1}^1=\int_{-1}^{1}h''''h-[h'''h]_{-1}^1\\
&\le& \|h''''\|_{L^2(-1,1)}\|h\|_{L^2(-1,1)}\le C\|h''''\|_{L^2(-1,1)}\|h''\|_{L^2(-1,1)}
\end{eqnarray*}
where the boundary terms vanish in both cases and the last inequality follows from the above mentioned estimate $\|h\|_{L^2(-1,1)}\le C\|h''\|_{L^2(-1,1)}$.\par
For the second bound, we need to distinguish the two cases. If $h''(\pm1)=0$, then an integration by parts shows that
$$
\|h'''\|_{L^2(-1,1)}^2 = -\int_{-1}^{1}h''''h''+[h'''h'']_{-1}^1\le\|h''''\|_{L^2(-1,1)}\|h''\|_{L^2(-1,1)}\le C\|h''''\|_{L^2(-1,1)}^2
$$
where the boundary term vanishes because $h''(\pm1)=0$ and the last inequality follows from the previously proved estimate.\par
If $h'(\pm1)=0$, we give a simple one-dimensional proof by repeatedly using the Rolle Theorem. From $h(\pm1)=0$ we infer that there exists $y_0\in(-1,1)$ such that $h'(y_0)=0$.
Subsequently, there exist $y_1\in(-1,y_0)$ and $y_2\in(y_0,1)$ such that $h''(y_1)=h''(y_2)=0$. Finally, there exists $y_3\in(y_1,y_2)$ such that $h'''(y_3)=0$.
Then, by the Fundamental Theorem we get
$$
h'''(y)=\int_{y_3}^{y}h''''(t)dt
$$
and we conclude using the H\"{o}lder inequality.\end{proof}

\subsection*{Acknowledgments}
VB, EB and FG  were partially supported by the 2023-27 Excellence Department Grant of MUR (Italian Ministry) and by the GNAMPA group of the Istituto Nazionale di Alta Matematica (INdAM). VB and FG were also partially supported by the PRIN project {\it Partial differential equations and related geometric-functional inequalities}.
EB was also partially supported by the Horizon Europe EU Research and Innovation Programme through the Marie Sklodowska-Curie THANAFSI Project No. 101109475.

\subsection*{Conflicts of Interest} We confirm that we do not have any conflict of interest.

\subsection*{Data Availability} The manuscript has no associated data.

{\small
\bibliographystyle{emsplain}

}
\end{document}